\documentclass[11pt]{article}
\usepackage{amsmath,amsfonts,amsthm,amscd,amssymb,graphicx}
\usepackage{subfigure}

\numberwithin{equation}{section}

\usepackage{hyperref}
\usepackage{verbatim} 
\usepackage{color}
\usepackage{stackengine,scalerel}

\newcommand\obullet[1]{\ThisStyle{\ensurestackMath{%
  \stackon[1pt]{\SavedStyle#1}{\SavedStyle\kern.6\LMpt\bullet}}}}
\newcommand\ocirc[1]{\ThisStyle{\ensurestackMath{%
  \stackon[1pt]{\SavedStyle#1}{\SavedStyle\kern.6\LMpt\circ}}}}

\newtheorem{theorem}{Theorem}[section]

\newtheorem{lemma}[theorem]{Lemma}

\newtheorem{proposition}[theorem]{Proposition}

\newtheorem{corollary}[theorem]{Corollary}

\newtheorem{remark}[theorem]{Remark}

\def\eps{\varepsilon }

\renewcommand{\div}{{\rm div}}

\newcommand{\Be}{\begin{equation}}
	\newcommand{\Ee}{\end{equation}}

\def\beq{\begin{equation}}
\def\eeq{\end{equation}}
\def\bb1{{1\!\!1}}
\def\R{\mbox{Re }}

\def\w{{\omega}}

\def\pt{\partial}

\def\eps{\varepsilon}

\def\triangle{\Delta}
\def\bega{\begin{aligned}}
\def\enda{\end{aligned}}
\def\R{\mathbb{R}^2}

\def\lw{\left}
\def\rw{\right}

\def\R{\mathbb{R}}
\def\wtd{\widetilde}
\def\la{\langle}
\def\ra{\rangle}

\def\DF{\nabla_x}
\def\ta{\theta}
\def\wlm{\rightharpoonup}
\def\bcase{\begin{cases}}
\def\ecase{\end{cases}}
\def\al{\alpha}
\def\bmx{\begin{bmatrix}}
\def\emx{\end{bmatrix}}
\def\lbb{L^4\cap L^{4/3}}

\begin{document}

\title{
Asymptotics of Helmholtz-Kirchhoff 
Point-Vortices in the Phase Space
} 
\author{
Chanwoo Kim, Trinh T. Nguyen\footnotemark[1]}
\maketitle

\renewcommand{\thefootnote}{\fnsymbol{footnote}}

\footnotetext[1]{Department of Mathematics, University of Wisconsin-Madison, WI 53706, USA. Emails: chanwookim.math@gmail.com; 
tnguyen67@wisc.edu.}
\begin{abstract} A rigorous derivation of point vortex systems from kinetic equations has been a challenging open problem, due to singular layers in the inviscid limit, giving a large velocity gradient in the Boltzmann equations.  In this paper, we derive the Helmholtz-Kirchhoff point-vortex 
system from the hydrodynamic limits of the Boltzmann equations. We construct Boltzmann solutions by the Hilbert-type expansion associated to the point vortices solutions of the 2D Navier-Stokes equations. We give a precise pointwise estimate for the solution of the Boltzmann equations with small Strouhal number and Knudsen number. \end{abstract}

\tableofcontents

\section{Introduction}
One of the central open questions in the field of partial differential equations concerns Hilbert's sixth problem \cite{HilbertICM}, which aims to develop a comprehensive theory of gas dynamics and bridge the gap between the mesoscopic Boltzmann equations and the macroscopic fluid models in formal limits. The Boltzmann equation serves as a fundamental model of kinetic theory for dilute collections of gas particles in elastic binary collisions. It is remarkable that the fundamental equations of fluid dynamics, which describe compressible and incompressible fluids as well as those that are either inviscid or viscous, can all be derived from the Boltzmann equation for rarefied gas dynamics. This can be achieved by selecting the appropriate scalings in the limit of small mean-free path. To be more precise, one considers the Boltzmann equation with scales $\eps,\kappa>0$:
\beq\label{F-eq}
\eps\pt_t F+v\cdot\nabla_x F=\frac{1}{\eps\kappa}Q(F,F),
\eeq
where $F$ is the probability distribution function 
\[\bega 
F: (0,\infty)\times \R^2\times \R^3&\to (0,\infty)\\
(t,x,v)&\to F(t,x,v).
\enda 
\]Here $v\cdot \nabla_x =v_1\pt_{x_1}+v_2\pt_{x_2}$, and $F$ is $x_3$ independent. The bilinear term $Q(F,F)$ in \eqref{F-eq} is the collision term describing the interactions of particles that are uncorrelated before the collision. In this paper, we consider the hard-sphere Boltzmann collision operator
\[\bega 
Q(F,G)(v)=\frac 1 2 \int_{\R^3\times \mathbb S^2}|(v-v_\star)\cdot \wtd s|&\{F(v')G(v_\star')+G(v')F(v_\star')\\
&-F(v)G(v_\star)-G(v)F(v_\star)
\}d\wtd sdv_\star,
\enda 
\]
where 
\[
v'=v-((v-v_\star)\cdot\wtd s)\wtd s,\qquad v_\star'=v_\star+((v-v_\star)\cdot \wtd s)\wtd s,\qquad \wtd s\in \mathbb S^2.
\]
The equation \eqref{F-eq} has the global Maxwellian 
\beq\label{global}
\mu(v)=\frac{1}{(2\pi)^{3/2}}e^{-\frac{|v|^2}{2}}
\eeq 
as an exact steady solution. 
When $\eps\to 0,\kappa\to 0$, one expects the famous Hilbert asymptotic expansion 
\[
F=\mu+\eps\sqrt{\mu}f_1+\eps^2 \sqrt\mu f_2+o(1) \]
where $f_1(t,x,v)=u^\kappa(t,x)\cdot (v_1,v_2)$ and $u^\kappa$ solves the Navier-Stokes equations 
\beq\label{NSeq}
\bega
\pt_t u^\kappa+u^\kappa\cdot\nabla_x u^\kappa+\nabla_x p^\kappa&=\kappa\eta_0\triangle u^\kappa, \\
\nabla\cdot u^\kappa&=0.
\enda
\eeq
where $\eta_0>0$ is a physical constant that can be computed explicitly by the Boltzmann theory. 
This connection between the Boltzmann equations and the fluid equations was first discovered by Hilbert \cite{HilbertICM}, which is now the famous Hilbert's sixth problem, concerning the hydrodynamics limits of the Boltzmann equations. The weak convergence of the moments of $F$ was formally derived in \cite{bardos-golse-levermore1} and is finally justified for DiPerna-Lions renormalized solutions in \cite{Golse-Laure} (see also \cite{bardos-golse-levermore2,bardos-golse-levermore3,bardos-golse-levermore4,Lions-Masmoudi,Masmoudi-SR,Saint-Raymond-book}). We refer the readers to \cite{Caflisch1,Lachowicz,Nishida,UkaiAsano} concerning the closeness of the Hilbert expansion of the Boltzmann equation to the solutions of the compressible Euler equations, \cite{YuCPAM} for shock wave solutions, \cite{HWT-discontinuity} around the Euler equations with contact discontinuities, \cite{KimLa} for initial Euler vorticity below the Yudovich class, \cite{Speck-Strain} for the relativistic fluids, \cite{Esposito-Guo-Kim-Marra1,Esposito-Guo-Kim-Marra2,ELM1,ELM2,AEMN1,MR2765739} for the steady Boltzmann equations,  \cite{YanGuoCPAM06} for the Hilbert expansion on the torus, and \cite{JangKimAPDE,CaoJangKim} on the half-space.~\\~\\
 On the other hand, the fluid motions are governed by the famous Navier-Stokes equations with typically large Reynolds numbers. Two-dimensional fluids with irregular data, such as vortex patches or point vortices, turn out to be a very good approximation to real-world fluids, such as air flows, vortices clouds, and hurricanes. One significant difficulty is the appearance of singular layers, which can create large convection term in the inviscid limit $\kappa\to 0$. We refer the readers to  \cite{CW95,CW96,Chemin96, Sueur15,CDE22} on the inviscid limit for vortex patches, \cite{Gallay2010} for point vortices, and \cite{2NVW} for vortex-wave and  \cite{gallay2023vanishing} for vortex rings. Vortex structure and vortex merging play a key role in understanding coherent structures in turbulence theory, which generate considerable interest among physicists and mathematicians (\cite{Helmholtz1858,LinVortex,merge1,merge2,NewtonBook,Machiorobook,Gallay2010,GlassFranck18}).   It is observed that for flows with large Reynolds numbers, the vortex structure remains for a very long time, as long as the centers of the vortices remain separated. When two vortices are close to each other, they tend to merge into larger structures, and the shape of the vortices changes significantly (\cite{GH87,mokry_2005,reinaud_dritschel_2005,MZ82,melander_zabusky_mcwilliams_1988}). Moreover, it is known that for any fixed viscosity $\kappa$, the Lamb-Oseen vortex will govern the long time behavior of the 2D Navier-Stokes equations (see \cite{GallayWayne} and references therein). It is therefore important to have a unified theory that connects the kinetic Boltzmann equations and the singular vortex systems. One way to study these phenomena is to consider the incompressible Euler equations on the whole space $\R^2$, which is given by setting the viscosity $\kappa$ to zero: 
 \beq
 \pt_t u^0+u^0\cdot\nabla u^0+\nabla p^0=0,\qquad \nabla\cdot u^0=0.
 \eeq
 The vorticity $\w^0=\nabla\times u^0$ then solves the transport equation 
 \beq\label{Euler}
 \pt_t \w^0+u^0\cdot\nabla \w^0=0.
 \eeq
 It is natural to consider the initial data of the Euler equations \eqref{Euler} to be the sum of $N$ point vortices $\{z_{i,0}\}_{i=1}^N$ with circulations $\{\al_i\}_{i=1}^N$. Namely
 \beq\label{data}
 \w^0(x)|_{t=0}=\sum_{i=1}^N \al_i  \delta_0(x-z_{i,0}),
 \eeq
 where $\delta_{0}$ denotes the Dirac delta function at $0\in \R^2$. Indeed, the initial data \eqref{data} is too singular to define any notion of weak solutions to Euler, and the existence of weak solutions to the Euler equations with initial data \eqref{data} remains open.
 In \cite{Marchioro1,Marchioro2,Machiorobook},  by regularizing the vortex data  \eqref{data} under certain regimes, Machioro and Machioro-Pulvirenti showed that the solution to the corresponding Navier-Stokes converges weakly to
 \[
 \sum_{i=1}^N \al_i  \delta_0(x-z_{i}(t)), 
 \]
 where  the point vortices $\{z_i(t)\}_{i=1}^N$ solve the Helmholtz-Kirchhoff point-vortex 
system (\cite{Helmholtz1858}):
\beq\label{HK}
\begin{cases}
z_i'(t)&=\sum_{j\neq i} \alpha_j K_{B}(z_i(t)-z_j(t)),\\
z_i(0)&=z_{i,0}.
\end{cases}
\eeq
Here $K_B(x)$ is the kernel given by 
\beq\label{KB}
K_B(x)=\frac 1 {2\pi}\frac{x^\perp}{|x|^2},\qquad x^\perp=(-x_2,x_1).
\eeq
Formally speaking, the velocity of the vortex $z_i(t)$ is given by the velocity generated by $\sum_{j\neq i} \al_j \delta_0(x-z_{j,0})$, which are all the other vortices disregarding the self-interaction.  
In \cite{Gallay2010}, Gallay shows that the vorticity $\w^\kappa$ of the Navier-Stokes equations 
\beq\label{vor}
\pt_t\w^\kappa+u^\kappa\cdot\nabla\w^\kappa=\kappa\eta_0\triangle \w^\kappa,\qquad u^\kappa=K_B\star_{\R^2}\w^\kappa,\eeq
with singular initial data \eqref{data}
converges strongly in $L^1(\R^2)$ to the sum of Gaussians, called Lamb-Oseen vortices, centered at the point vortices $\{z_i(t)\}_{i=1}^N$: \[
\sum_{i=1}^N \frac{\al_i}{4\pi \kappa\eta_0 t}e^{-\frac{|x-z_i(t)|^2}{4\kappa\eta_0 t}},
\]
for any finite time $T>0$ such that there exists a $d_T>0$ satisfying \beq \label{gap1}
\min_{i\neq j}|z_i(t)-z_j(t)|\ge d_T>0\qquad\text{for all}\qquad t\in [0,T].
\eeq
We refer the readers to \cite{cottet,GigaNS,KatoNS,GallayNS} on the existence and uniqueness of the 2D Navier-Stokes equations with such initial data, or in fact more generally, with initial data of finite measures. It is known that when all $\al_i$ have the same sign, the system \eqref{HK} has a global solution, and the vortices never collapse (\cite{Machiorobook}).~\\~\\
 The goal of this paper is to justify the point-vortex system \eqref{HK} from the Boltzmann equation in the hydrodynamic limits of $F=F^{\eps,\kappa}$ when $(\eps,\kappa)\to 0$. Our main results are stated in Theorem \ref{mainthm}.
\subsection{Main results}
\begin{theorem}\label{mainthm} Let $\mu$ be the global Maxwellian given in \eqref{global}, $K_B$ be the Biot-Savart kernel \eqref{KB} and $d_T$ be the constant defined in \eqref{gap1}.
There exists a local solution $F(t,x,v)\in C([0,T],L^\infty(\R_x^2\times \R_v^3))$ to the Boltzmann equation \eqref{F-eq} so that for any $\gamma_0\in (1,2)$ and $\rho_1\in (0,\frac 1 4)$, there exist constants $C_0,\rho_0>0$ and  independent of $\kappa,\eps$ such that if
\beq\label{epskappa}
\eps \exp \exp \exp\lw (C_0\exp(\kappa^{-1})\rw)\le  1,
\eeq
Then 
\[\bega 
&\lw|
F(t,x,v)-\mu(v)-\eps  {\mu}(v) v\cdot \sum_{i=1}^N \al_i K_B(x-z_i(t))
\rw|\\
&\le C_0 \lw\{\frac{\eps}{d(t,x)}\min\lw\{1,\frac{e^\frac{-d_T^2}{C_0\kappa}}{d(t,x)}
\rw\}+\frac{\eps}{d(t,x)}e^{-\frac{d(t,x)^2}{16\kappa (t+1)}}+\sqrt\kappa \eps 
+\eps^{\gamma_0} \rw\}e^{-\rho_1|v|^2}\\
\enda 
\]
for all $t\in [0,T]$, as long as \[
d(t,x)=\min_{1\le i\le N}|x-z_i(t)|\ge 2C_0 e^{-\frac{d_T^2}{C_0\kappa}}.
\] 

\end{theorem}
\begin{remark}Theorem \ref{mainthm} gives a precise pointwise estimate for the solution of the Boltzmann solutions for any $x$ away from the point vortices $\{z_i(t)\}_{i=1}^N$. This is natural, due to the singular nature of the velocity field $K_B(x-z_i(t))$ at the point vortices. More precisely, if $x$ is away from the point vortices with $\kappa$, then clearly the Helmholtz-Kirchhoff point vortices are the leading order term.
\end{remark}
\begin{corollary}
Under the same assumption, there holds 
\[
\frac{F(t,x,v)-\mu(v)}{\eps}\to \sqrt{\mu(v)} v\cdot \sum_{i=1}^N \al_i K_B(x-z_i(t)),
\]
uniformly in any compact set $C\subset\R^2$ as $(\eps,\kappa)\to 0$, given $\text{dist}\lw(C,\cup_{i=1}^N\{z_i(t):0\le t\le T\}\rw)>0$.
\end{corollary}
\begin{theorem} \label{macro-vor}
There exists a local solution $F(t,x,v)$ to the Boltzmann equation \eqref{F-eq} so that the microscopic vorticity
\beq\label{vorBE}
\w^\eps(t,x)=\frac{1}{\eps}\int_{\R^3}\lw(v_1\pt_{x_2}F-v_2\pt_{x_1}F
\rw)dv
\eeq 
converges weakly to
\[
\sum_{i=1}^N \alpha_i  \delta_0 (x-z_i(t)),
\]
as $\kappa\to 0$ and $\eps\to 0$. 
\end{theorem}

\begin{remark} The condition \eqref{epskappa} can be replaced by $\eps \exp\exp(C_0\exp(\kappa^{-1}))\le 1$ if the pointwise convergence is replaced by $L^2_{x,v}$ convergence, by the same analysis, without the need for a third-order approximation. Here, we are able to justify the expansion in the $L^\infty$ space by a higher order Hilbert expansion. It turns out that the next order term will have an extra exponential growth in $\kappa^{-1}$, which is natural from the convected heat and linearized Navier-Stokes equations (see Proposition \ref{f1f2} and Proposition \ref{deri-thm2}), giving the condition \eqref{epskappa}.
\end{remark}~\\
\textbf{Acknowledgment.} TN is partly supported by the AMS-Simons Travel Grant Award. CK is partly supported by NSF-DMS 1900923, NSF-CAREER 2047681, the Simons fellowship in Mathematics, and the Brain Pool fellowship funded by the Korean Ministry of Science and ICT. He thanks Joonhyun La and In-Jee Jeong for stimulating discussion. 
\subsection{Difficulties and main ideas}

Using the famous Hilbert expansion \cite{HilbertICM}, we write the solution $F$ of the equation \eqref{F-eq} as
\[
F=\mu+\eps\sqrt{\mu}f_1+\eps^2 \sqrt\mu f_2+\eps^3 \sqrt\mu f_3+\eps \delta \sqrt\mu f_R.
\]
The main aim in the hydrodynamic limit is proving that $\mu+\eps\sqrt\mu f_1$ is the leading order in the above expansion locally. Let  us define
\beq\label{fa-intro}
f_a=\sqrt\mu f_1+\eps^2\sqrt\mu f_2+\eps^3\sqrt\mu f_3
\eeq
to be the approximate solution, then one can write the equation for the remainder $f_R$ as 
\beq\label{fR-intro}
\bega 
\pt_t f_R+\frac 1 \eps v\cdot\nabla_x f_R+\frac{1}{\kappa\eps^2}Lf_R&=\frac{2}{\kappa\eps^2}\Gamma(f_a,f_R)+\frac{\delta}{\kappa\eps}\Gamma(f_R,f_R)+R_{B,a}(f_a)\\
\enda
\eeq
where 
\beq\label{error-intro}
R_{B,a}(f_a)=-\frac{1}{\eps\delta}\lw\{\pt_t f_a+\frac 1 \eps v\cdot\nabla_x f_a+\frac{1}{\kappa\eps^2}Lf_a-\frac{1}{\kappa\eps^2}\Gamma(f_a,f_a)
\rw\}.
\eeq
There are two main difficulties we need to overcome, which can be summarized as follows:
\begin{itemize}\item{
Firstly, one needs to \textit{construct a good approximation solution built from the Helmholtz-Kirchhoff point vortex system}. This approximation is highly nontrivial, since the point velocity, formally given as $\frac{x^\perp}{|x|^2}$, is too singular at $x=0$. At the same, time, the vorticity associated with each vortex is exactly the Dirac delta function placed at the vortex.
 }\item{Secondly, even if one can construct the higher approximations to the Boltzmann equations to make $R_{B,a}=O(\eps^M)$ for any large $M$, \textit{controlling the exact remainder} $f_R$ is the key difficulty, due to linear term $\frac{2}{\kappa\eps^2}\Gamma(f_a,f_R)\sim \frac{1}{\kappa \eps}\Gamma(f_1,f_R)$ and the nonlinear term $\frac{\delta}{\kappa\eps}\Gamma(f_R,f_R)$.}
\end{itemize}
These two difficulties are in fact very similar to the justification of the Prandtl boundary layers expansion for the Navier-Stokes equations, while one first needs to show that the Prandtl is well-posed and build a good approximation solution, then one needs to control the exact remainder, where the leading linear term is large (due to the boundary layers in the approximate solution) and the loss of derivatives in the nonlinear term (see for instances, \cite{NNArma,2N-ext,2N1,ITFV,Trinh-cri,NguyenGrenierAPDE}). Below, we give the main ideas on how we overcome each difficulty mentioned above. For general audience, we explain them in the process of the Hilbert expansion with some details, where the expansion itself had been already established by the first author and collaborators \cite{JangKimAPDE, CaoJangKim}.  

\medskip

Plugging the approximate solution $f_a$ in \eqref{fa-intro} into the error term $R_{B,a}(f_a)$ given in \eqref{error-intro} and collecting the same order terms, we obtain
\beq\label{RB-intro}
\bega 
R_{B,a}(f_a)&=-\frac{1}{\eps\delta}\lw\{\frac 1 {\kappa\eps}Lf_1\rw\}\\
&\quad-\frac{1}{\eps\delta}\lw\{v\cdot\nabla_x f_1 +\frac{1}{\kappa}Lf_2-\frac 1 \kappa \Gamma(f_1,f_1)
\rw\}\\
&\quad-\frac {1}{\eps\delta}\cdot \eps \lw\{\pt_t f_1+v\cdot\nabla_x f_2+\frac 1 \kappa Lf_3-\frac 2 \kappa \Gamma(f_1,f_2)
\rw\}\\
&\quad-\frac {1}{\eps\delta}\cdot \eps^2 \lw\{\pt_t f_2+v\cdot\nabla_x f_3-\frac{1}{\kappa}\lw(\Gamma(f_2,f_2)+2\Gamma(f_1,f_3)
\rw)
\rw\}\\
&\quad-\frac 1 {\eps\delta}\cdot \eps^3 \lw\{\pt_t f_3-\frac 2 \kappa \Gamma(f_2,f_3)-\frac{\eps}{\kappa}\Gamma(f_3,f_3)
\rw\}.
\enda 
\eeq
Here, $L$ is the linear operator and $\Gamma$ is a bilinear term (see the precise definitions in Section \ref{BE-con}). From the above expression of $R_{B,a}(f_a)$, it is natural to solve the following system of equations
\beq\label{sys-intro}
\begin{cases}
&Lf_1=0,\\
&v\cdot\nabla_x f_1 +\frac{1}{\kappa}Lf_2-\frac 1 \kappa \Gamma(f_1,f_1)=0,\\
&\pt_t f_1+v\cdot\nabla_x f_2+\frac 1 \kappa Lf_3-\frac 2 \kappa \Gamma(f_1,f_2)=0,\\
&P(\pt_t f_2+v\cdot\nabla_x f_3)=0.
\end{cases}
\eeq
The set of equations \eqref{sys-intro} will determine our conditions on  $f_1,f_2$ and $f_3$. Here, the $Pf$ denotes the projection of $f \in L^2(\R_v^3)$ onto the null space $\mathcal N=\ker(L).$ From the first two equations in \eqref{sys-intro} (see our detailed calculations in Section \ref{H-exp}), we compute $f_1$ and the microscopic part of $f_2$ as follows:
\beq\label{f2-intro}
\begin{cases}
&f_1=u^\kappa\cdot v\sqrt\mu, \\
&(I-P)f_2=\frac 1 2 \sum_{i,j}u_i^\kappa u_j^\kappa \hat A_{ij}(v)-\kappa\sum_{i,j}\pt_j u_i^\kappa A_{ij}(v).
\end{cases}
\eeq
Here, $(I-P)f_2$ is the projection of $f_2$ onto $(\ker L)^\perp$ and $\hat A_{ij}(v), A_{ij}(v)$ are good functions decaying in $v$ (see the definition \eqref{Aij}). This is why $f_1$ is called the fluid part, as it only lives in the macroscopic space $\mathcal N$.
By projecting the second equation in \eqref{sys-intro} onto $\mathcal N$,   we see that $u^\kappa$ \textit{must be the solution to the Navier-Stokes equations}. This is crucial to kill the large factor $-\frac 1 \delta$ in the third line of \eqref{RB-intro}. Since we expect the leading order to be $\mu+\sqrt\mu f_1$, our $\delta$ must be at least $\eps^{1+}$, which means that any higher order approximation in $\kappa$ at the fluid level will not work, as $\eps\ll \kappa$. Now, let us detail why this will be a difficulty on the Boltzmann equation, namely the approximate solution $f_a$.
Since we expect $u^\kappa(t)$ to behave like the regularized solution,
\[
u^\kappa(t,x)\sim \sum_{i=1}^N \al_i K_B(x-z^\kappa_i(t))\lw(1-e^{-\frac{|x-z^\kappa_i(t)|^2}{4\eta_0\kappa t}}
\rw),
\]
we have 
\beq\label{sing}
\|\nabla_x u^\kappa(t)\|_{L^\infty}\sim \frac 1 {\sqrt{\kappa t}}.
\eeq
Looking at the microscopic part of $f_2$ in \eqref{f2-intro}, one can see that there is a bilinear term $u_i^\kappa u_j^\kappa$ and a loss of derivative in $x$ term $\kappa \pt_j u_i^\kappa$, which is of order $t^{-\frac 1 2}$. The same mechanism happens in higher hierarchy, where the next order has more quadratic terms and is losing more derivatives in both $t$ and $x$ (\cite{YanGuoCPAM06}). This phenomenon is very different from constructing the Navier-Stokes equations. To be more precise, from the third equation in \eqref{sys-intro}, we have \[
(I-P)f_3=L^{-1}(2\Gamma(f_1,f_2)-\kappa(\pt_t f_1+v\cdot\nabla_x f_2)).
\]
Since we expect $\nabla_x u^\kappa\sim \frac 1 {\sqrt{\kappa t}}$ and $\pt_t u^\kappa\sim t^{-\frac 3 2}$ at least, this build-up of time singularity is far from time-integrable when controlling the remainder $f_R$. 
This time singularity also occurs in the equations for the macroscopic parts as a forcing term (see equations \eqref{c2} and \eqref{b2}). Even with the fact that the linear operator \[
\frac 1{\kappa\eps^2} L\]
on the left hand side of \eqref{fR-intro}
is coercive in the microscopic space $\mathcal N^\perp$,  there is no clear mechanism on how to use the gain in $\eps$ to kill the singularity in time.

On the other hand,  since the initial velocity $\sum_{i=1}^N \al_i K_B(x-z_{i,0})$ is not in  $L^2(\R_x^2)$, we can not take Fourier transform in $x$ and therefore, it is not clear whether one can construct the Boltzmann solutions with such singular velocity, using the Green function developed in \cite{MR2850554}. 

To overcome these difficulties, we smooth out the initial vortices by a viscous scale of size $\sqrt\kappa$ to avoid the initial layer $\sqrt{\kappa t}$. Formally speaking, we run the Boltzmann equations only after a time of order O(1) for point vortices of the Navier-Stokes equations. The key observation is that, on this time scale, the Navier-Stokes solutions still  preserve the vortex structures given by the point-vortex system \eqref{HK} (see Theorem \ref{NSthm1}). 
Our initial vorticity is given explicitly by 
\beq\label{initial-vor}
\w_0^\kappa(x)=\sum_{i=1}^N \frac{\al_i}{4\pi\kappa\eta_0}e^{-\frac{|x-z_{i,0}|^2}{4\kappa\eta_0}}\qquad \text{for all}\quad x\in\R^2,
\eeq
which corresponds to the initial velocity 
\beq\label{initial-vel}
u^\kappa_0(x)=K_{B}\star_{\R^2}\w_0^\kappa=\sum_{i=1}^N \frac{\alpha_i}{2\pi}\frac{(x-z_{i,0})^\perp}{|x-z_{i,0}|^2}\lw(1-e^{-\frac{|x-z_{i,0}|^2}{4\kappa\eta_0}}
\rw),\qquad x\in \R^2.
\eeq
By constructing an exact solution $u^\kappa$ solving the Navier-Stokes equations, with a precise pointwise asymptotic expansion 
\[
u^\kappa(t,x)=\sum_{i=1}^N \al_i K_B(x-z_i^\kappa(t))\lw(1-e^{-\frac{|x-z^\kappa_i(t)|^2}{4\kappa\eta_0(t+1)}}
\rw)+O(\sqrt{\kappa})_{L^\infty_x},
\]
where $z_i^\kappa(t)= z_i(t)+O\Big(e^{-\frac{d_T^2}{C_0\kappa(t+1)}}\Big)$, we are able to build $f_1$, which is the first order approximation of the Boltzmann equations.

\subsection{Organization of the paper}In section \ref{NS-sec}, we construct the Navier-Stokes solutions around the vortex system \eqref{HK}, and obtain the point-wise estimate on the velocity field and $L^1(\R_x^2)$ estimate on the vorticity. We also include subsection \ref{higher-sec} in the same section,  which is needed for the estimates on the Boltzmann solutions, when controlling the remainder $f_R$. In section \ref{BE-con}, we introduce briefly the standard definitions and properties from the Boltzmann equations, give estimates on $L^2_{t,x,v}$, $L^2_t L^\infty_{x,v}$ and $L^\infty_{t,x,v}$ for the Boltzmann solution. We give the proof of our main theorems in subsection \ref{proof-main}. 
 
 \section{Construction of Navier-Stokes solutions}\label{NS-sec}

We decompose the vorticity $\w^\kappa$  of the Navier-Stokes equation \eqref{vor} as 
\[
\w^\kappa(t,x)=\sum_{i=1}^N \al_i\w_i(t,x)
\]
where $\w_i$ solves the convected heat equation
\beq\label{omega-i}
\pt_t\w_i+u^\kappa\cdot\nabla_x\w_i=\kappa\eta_0\triangle\w_i,\qquad \w_i|_{t=0}=\frac{1}{4\pi\kappa\eta_0}e^{-|x-z_{i,0}|^2/4\kappa}
\eeq
where $u^\kappa=\sum_{i=1}^N  \al_i u^\kappa_i$, $u_i^\kappa=K_B\star_{\R^2}\w_i$. From \eqref{omega-i}, it is straightforward that 
\[
\frac{d}{dt}\int_{\R^2}\w_i(t,x)dx=0,
\]
and this implies the conservation of momentum for each $\w_i$:
\[
\int_{\R^2}\w_i(t,x)dx=1
\]
for all $t>0$ and $1\le i\le N$.
Our main theorem in this section is as follows
\begin{theorem} \label{NSthm1}
There exists a time $T_\star>0$ such that 
\[
\lw\|\w^\kappa (t,x)-\sum_{i=1}^N\frac{\al_i}{4\pi\kappa\eta_0(t+1)}e^{-\frac{|x-z_i(t)|^2}{4\kappa\eta_0(t+1)}}
\rw\|_{L^1(\R^2)}\lesssim \kappa 
\]
uniformly in $t\in [0,T_\star]$.
\end{theorem}
To describe the precise vortex structure of the Navier-Stokes equations, inspired from \cite{GallayWayne}, we  use the rescaled variables 
\[
X=\frac{x-z_i^\kappa(t)}{\sqrt{\kappa\eta_0(t+1)}},\qquad \tau=\ln(t+1)
\]
for each $i$ and the solution $\w_i$ solving \eqref{omega-i}.
For each $i\in \{1,\cdots,N\}$, we define 
\[
y_i^\kappa(\tau)=z_i^\kappa(t),\qquad y_i(\tau)=z_i(t),\qquad t=e^{\tau}-1.
\]
Here 
 $\{y_i^\kappa(\tau)\}_{i=1}^N$ solves the viscous vortex system 
\beq\label{viscous}
\begin{cases}
&\pt_\tau y^\kappa_i=e^{\tau}\sum_{j\neq i}\al_jK_{B}(y^\kappa_{ij}(\tau))\lw(1-e^{-\frac{|y^\kappa_{ij}(\tau)|^2}{4\kappa\eta_0 e^\tau}}
\rw),\\
&\qquad y_i^\kappa(0)=z_{i,0}.
\end{cases}
\eeq
where $y_{ij}^\kappa(\tau)=y_i^\kappa(\tau)-y_j^\kappa(\tau)$. 
It is straight forward that the viscous vortex system has a global solution for all initial data. In the new time variable $\tau\in [0,\infty)$, the equations for $\{y_i(\tau)\}_{i=1}^N$ reads 
\beq\label{viscous}
\begin{cases}
&\pt_\tau y_i=e^{\tau}\sum_{j\neq i}\al_jK_{B}(y_{ij}(\tau)),\\
&\qquad y_i(0)=z_{i,0}.
\end{cases}
\eeq
In Lemma \ref{gap-lem} in the appendix, we show that the viscous vortex approximation is close to the original vortices $\{y_i(\tau)\}_{i=1}^N$ as $\kappa\to 0$, as long as they are well separated \eqref{gap1}. To be more precise, we can choose $d_T>0$ and a time $\tau_\star>0$ such that the following inequalities hold:
\[
\begin{cases}
\min_{i\neq j}|y_i(\tau)-y_j(\tau)|&\ge d_T,\\
\min_{i\neq j}|y_i^\kappa (\tau)-y_j^\kappa (\tau)|&\ge d_T,\\
\max_{1\le i\le N}|y_i^\kappa(\tau)-y_i(\tau)|&\lesssim e^{-\frac{d_T^2}{4\kappa}}.
\end{cases}
\]
We also define the 
\[
G(X)=\frac 1 {4\pi} e^{-|X|^2/4},\qquad u^G(X)=\frac{1}{2\pi}\cdot\frac{X^\perp}{|X|^2}\lw(1-e^{-|X|^2/4}
\rw).
\]
In the rescaled variables $(\tau,X)$, we look for solutions of \eqref{omega-i} of the form
\[
\w_i(t,x)=\frac{1}{\kappa \eta_0e^\tau }\bar \w_i\lw(\tau,X
\rw),\qquad u_i(t,x)=\frac{1}{\sqrt{\kappa \eta_0e^\tau }}\bar u_i(\tau,X)\\
\]
and rewrite the equation  \eqref{omega-i} as follows:
\beq\label{wi-eq}
\pt_\tau \bar\w_i-\mathcal L\bar\w_i+\frac{1}{\kappa\eta_0}\bar u_i\cdot\nabla_X \bar\w_i+\frac{1}{\kappa\eta_0}\sum_{j\neq i}\al_j\lw\{\bar u_j\lw(\tau,X+\frac{z_{ij}^\kappa(\tau)}{\sqrt{\kappa \eta_0 e^\tau}}
\rw)-u^G\lw(\frac{z^\kappa_{ij}(\tau)}{\sqrt{\kappa \eta_0 e^\tau}}
\rw)
\rw\}\cdot \nabla_X \bar\w_i=0.
\eeq
Here, the linear operator $\mathcal L$ is defined as 
\[
\mathcal L \bar \w_i=\triangle_X\bar\w_i+\frac 1 2 X\cdot\nabla_X \bar\w_i+\bar\w_i.
\]
We define the weighted $L^2$ space
\[
\|\w(\tau)\|_{L^2_p}=\lw(\int_{\R^2}|\w(\tau,X)|^2e^{|X|^2/4}dX \rw)^{1/2}
\]
for any function $\w$. Our goal is to show that $\w_i(\tau,X)$ converges to $G(X)$ in the space $L^2_p$ as $\kappa\to 0$. The precise  theorem is as follows:
\begin{theorem}\label{NSthm2}
There exists a time $\tau_\star>0$ such that 
\[
\sup_{0\le \tau\le\tau_\star}\|\bar \w_i(\tau)-G(X)\|_{L^2_p}\to 0
\]
as $\kappa\to 0^+$. 
\end{theorem}
\begin{remark} It is straightforward that Theorem \ref{NSthm2} implies Theorem \ref{NSthm1}. As a consequence, we have the weak convergence 
\[
\w^\kappa(t,x)\wlm \sum_{i=1}^N \al_i \delta_{z_i(t)}
\]
as $\kappa\to 0$.
\end{remark}~\\
To show Theorem \ref{NSthm2}, we establish the expansion 
\beq\label{exp}
\begin{cases}
\bar\w_i(\tau,X)&=G(X)+\kappa\eta_0  \wtd w_{i,a}+\kappa \eta_0 w_i,\\
\bar u_i(\tau,X)&=u^G(X)+\kappa\eta_0  \wtd u_{i,a}+\kappa\eta_0 u_i,
\end{cases}
\eeq
where $\wtd w_{i,a}$ is the next order approximation that will be constructed in the next section, improving the error induced by $G(X)$, and $w_i$ is the exact remainder. Here, the velocities $\wtd u_{i,a}$ and $u_i$ are obtained from the vorticities $\wtd w_{i,a}$ and $w_i$ by Biot-Savart law.

\subsection{Next order approximation}
Putting $\bar \w_i=G(X)$ and $\bar u_i=u^G(X)$ into the equation \eqref{wi-eq}, we obtain the following error
\[
R_i(\tau,X)=\sum_{j\neq i}\al_j R_{ij}(\tau,X)
\]
where 
\[
R_{ij}(\tau,X)=\frac{1}{\kappa\eta_0}\lw(u^G\lw(X+\frac{y_{ij}^\kappa(\tau)}{\sqrt{\kappa\eta_0 e^\tau}}
\rw)-u^G\lw(\frac{y^\kappa_{ij}(\tau)}{\sqrt{\kappa\eta_0 e^\tau}}
\rw)
\rw)\cdot\nabla_X G.
\]
In the next lemma, we expand $R_{ij}$ in terms of $\kappa$. 
\begin{lemma}\label{Rlem}
There holds 
\[\bega
R_{ij}(\tau,X)&=T_{ij}(\tau,X)+S_{ij}(\tau,X)
\enda\]
where
\beq\label{Aij}
\bega
T_{ij}(\tau,X)&=\frac{e^{-|X|^2/4}}{16\pi^2}\lw\{\frac{e^\tau}{|z_{ij}^\kappa(\tau)|^2}|X|^2\sin(2\psi_{ij})-\sqrt{\kappa\eta_0}\frac{e^{3\tau/2}}{|z_{ij}^\kappa(\tau)|^3}|X|^3\sin(3\psi_{ij})\rw\},\\
S_{ij}(\tau,X)&\lesssim \kappa e^{-\gamma|X|^2/4}
\enda 
\eeq
for any $\gamma\in (0,1)$, where $\psi$ is the angle between $X$ and $\eta_{ij}=\frac{y_{ij}^\kappa(\tau)}{\sqrt{\kappa e^\tau}}$, determined by the formula 
\[
\sin(\psi_{ij})=\frac{X^\perp\cdot \eta_{ij}}{|X||\eta_{ij}|}.
\]
\end{lemma}
\begin{proof} We consider 2 cases:~\\
\textbf{Case 1.} $|X|\ge \frac{d_T}{2\sqrt{\kappa\eta_0 e^\tau}}.$~\\
In this case, we use the fact that $\nabla_X  G=-\frac 1 2 XG(X)$ and $\|u^G\|_{L^\infty_X}\lesssim 1$ to get 
\[
R_{ij}(\tau,X)\lesssim \frac{1}{\kappa}\|u^G\|_{L^\infty} |X|e^{-|X|^2/4}\lesssim e^{-\gamma |X|^2/4}.
\]
\textbf{Case 2.} $|X|\le \frac{d_0}{2\sqrt{\kappa\eta_0 e^\tau}}$.~\\
Let $\eta_{ij}=\frac{z^\kappa_{ij}(\tau)}{\sqrt{\kappa \eta_0e^\tau}}$, then $R_i$ can be written as 
\[
R_i(\tau,X)=\frac 1 {\kappa\eta_0} \sum_{j\neq i} \lw(u^G(X+\eta_{ij})-u^G(\eta_{ij})
\rw)\cdot \nabla_X G.
\]
Now we have 
\[
2\pi\lw(u^G(X+\eta_{ij})-u^G(\eta_{ij})\rw)=\lw\{\frac{(X+\eta)^\perp}{|X+\eta|^2}-\frac{\eta^\perp}{|\eta|^2}\rw\}-\frac{(X+\eta)^\perp}{|X+\eta|^2}e^{-|X+\eta|^2/4}+\frac{\eta^\perp}{|\eta|^2}e^{-|\eta|^2/4}.
\]
The last two terms in the above expression are extremely small, due to the fact that  $|X+\eta|\ge |\eta|-|X|\ge \frac 1 2 |\eta|\ge \frac 1 2 \frac{d_T}{\sqrt{\kappa\eta_0 e^\tau}}$ and $|\eta|\ge \frac{d_T}{\sqrt{\kappa\eta_0 e^\tau}}$.
Hence we shall only consider 
\[
\frac{1}{\kappa\eta_0}\lw\{\frac{(X+\eta_{ij})^\perp}{|X+\eta_{ij}|^2}-\frac{\eta_{ij}^\perp}{|\eta_{ij}|^2}\rw\}\cdot\nabla_X G
\]
with $ |X|\le \frac 1 2 |\eta_{ij}|$. In this case, we have 
\[\bega 
&\lw(\frac{(X+\eta_{ij})^\perp}{|X+\eta_{ij}|^2}-\frac{\eta_{ij}^\perp}{|\eta_{ij}|^2}\rw)\cdot X=\eta_{ij}^\perp\cdot X\lw(\frac{1}{|X+\eta_{ij}|^2}-\frac{1}{|\eta_{ij}|^2}
\rw)=\sum_{n=2}^\infty (-1)^{n+1}\frac{|X|^n}{|\eta_{ij}|^n}\sin(n\psi_{ij})\\
&=(-1)|X|^2\frac{\kappa \eta_0 e^\tau}{|y_{ij}^\kappa(\tau)|^2}\sin(2\psi_{ij})+|X|^3\frac{(\kappa \eta_0e^\tau)^{\frac 3 2}}{|y_{ij}^\kappa(\tau)|^3}\sin(3\psi_{ij})+O\lw(|X|^4\frac{(\kappa\eta_0 e^\tau)^2}{|y_{ij}^\kappa(\tau)|^4}
\rw),
\enda 
\]
where $\psi_{ij}$ is the angle between $X$ and $\eta_{ij}=\frac{y_{ij}^\kappa(\tau)}{\sqrt{\kappa\eta_0 e^\tau}}$. The proof is complete.
\end{proof}~\\
To kill the $O(1)$ term in the remainder $R_{i}$ coming from the approximation $(w_i,v_i)=(G,u^G)$, we construct the second approximation 
\beq\label{omega-app}
\begin{cases}
\bar\w_{i,app}&=G(X)+\varphi(\tau)\kappa \eta_0 w_{i,a}(\tau,X),\\
 \bar u_{i,app}&=u^G(X)+\varphi(\tau)\kappa\eta_0 u_{i,a}(\tau,X),\\
 u_{i,a}&=K\star_X w_{i,a}.\\
 \end{cases}
\eeq
Here, $\varphi:[0,\infty)\to [0,\infty)$ is a cut-off function defined by 
\beq\label{cutoff}
\varphi(\tau)=\begin{cases}
&0\qquad \text{if}\qquad \tau\le (\eta_0\kappa)^{\alpha},\\
&1\qquad \text{if}\qquad \tau\ge 2(\eta_0\kappa)^\al.
\end{cases}
\eeq
where the parameter $\alpha$ is small, $\al\in (0,1)$. 
The purpose of the cut-off function $\varphi$ is to have 
\[
\bar w_{i,app}|_{\tau=0}=G(X), 
\]
which will give zero initial condition for the remainder \eqref{initial-wi}.
We also define 
\[
\wtd w_{i,a}(\tau,X)=\varphi(\tau)w_{i,a}(\tau,X),\qquad \wtd u_{i,a}(\tau,X)=\varphi (\tau) u_{i,a}(\tau,X).
\]
It is straight forward that $\wtd w_{i,a}|_{\tau=0}=0$ and $\wtd u_{i,a}=K_B \star_{\R^2}\wtd w_{i,a}$.
Plugging the above into the left hand side of \eqref{wi-eq}, we have the new error to be 
\beq\label{Rapp1}
\wtd R_{i,app}(\tau,X)=\varphi(\tau)R_{i,app}(\tau,X)+\kappa \eta_0 \varphi'(\tau)w_{i,a},
\eeq
where 
\beq\label{Rapp2}
\bega
R_{i,app}(\tau,X)&=\kappa\eta_0 (\pt_\tau-\mathcal L)w_{i,a}+\lw\{\Lambda w_{i,a}+T_i(\tau,X)\rw\}+S_i(\tau,X)\\
&\quad+\sum_{j\neq i} u_{j,a}\lw(\tau, X+\frac{y^\kappa_{ij}(\tau)}{\sqrt{\kappa e^\tau}}
\rw)\cdot\nabla_X G\\
&\quad+\kappa\eta_0\varphi(\tau)\sum_{j\neq i}u_{j,a}\lw(\tau,X+\frac{y^\kappa_{ij}(\tau)}{\sqrt{\kappa e^\tau}}
\rw)\cdot\nabla_X w_{i,a}+\kappa\eta_0\varphi(\tau)\lw( u_{i,a}\cdot\nabla_X w_{i,a}\rw).
\enda
\eeq
\begin{proposition}\cite{Gallay2010,GallayWayne}
There exists a solution $w_{i,a}$ to the problem 
\[
\kappa\eta_0(1-\mathcal L)w_{i,a}+\Lambda w_{i,a}=-T_i(\tau,X)
\]
such that \beq\label{wi-ineq}
|w_{i,a}(\tau,X)|+ |\nabla_X w_{i,a}(\tau,X)|+|\pt_\tau w_{i,a}(\tau,X)|\lesssim e^{-\gamma|X|^2/4}
\eeq
for any $\gamma\in (0,1)$. 
Moreover, $\int_{\R^2}w_{i,a}dX=0$.
\end{proposition}
\begin{proof} We sketch the proof here for the readers's convenience. 
Let $r=|X|$.  We can rewrite $T_{ij}(\tau,X)$ as a linear combination of the terms of the form 
$
r^2e^{-r^2/4}\sin(2\psi_{ij})
$
and $r^3 e^{-r^2/4}\sin(3\psi_{ij})$. It suffices to solve the equation 
\beq\label{lam1}
\kappa\eta_0(1-\mathcal L)w_n+\Lambda w_n=r^ne^{-r^2/4}\sin(n\theta),
\eeq
where $\theta$ is the angle between $X$ and $\eta_{ij}$ and $n\in \{1,2\}$. 
The solution is constructed via the iteration formula 
\beq \label{iter}
w_n=\Lambda^{-1}\lw\{r^n e^{-r^2/4}\sin(n\theta)
\rw\}+\kappa\eta_0\Lambda^{-1}\lw\{(1-\mathcal L)w_n
\rw\}=w_n^0+\kappa\eta_0\Lambda^{-1}(1-\mathcal L)w_n.
\eeq
\textbf{Construction of $\Lambda^{-1}$:} We look for solution $w_n^0$ of the form
\[
w^0_n=\triangle (\Omega^0_n(r)\cos(n\theta))=a_n^0(r)\cos(n\theta),
\]
where $\Omega^0_n(r)\cos(n\theta)$ is the stream function. This implies
\[
-\lw(\pt_r^2+\frac 1 r \pt_r -\frac{n^2}{r^2}\rw)\Omega_n^0=a_n^0(r).
\]
The velocity corresponding to $w_n^0$ is then given by $v_n^0=-\frac{n}{r}\Omega_n^0(r)\sin(n\theta)e_r+\pt_r\Omega_n^0(r)\cos(n\theta)e_\theta$. This implies 
\beq\label{lam2}
\Lambda w_n^0 =v_n^0\cdot\nabla_X G+v^G\cdot\nabla_X w_n^0=\lw\{\frac{n}{2}\Omega_n^0(r)\frac{1}{4\pi}e^{-r^2/4}-na_n^0(r)\frac{1-e^{-r^2/4}}{2\pi r^2}\rw\}\sin(n\theta).
\eeq
Combining \eqref{lam1} and \eqref{lam2}, we have the differential equation
\[
\pt_r^2\Omega_n^0+\frac 1 r \pt_r\Omega_n^0+\lw(\frac{r^2/4}{e^{r^2/4}-1}-\frac{n^2}{r^2}
\rw)\Omega_n^0=\frac{2\pi}{n}r^n e^{-r^2/4}.
\]
By a standard ODE theory, the solution can be constructed via the Green function for the homogeneous problem, with the decaying forcing on the right. The solution $w_n$ is constructed using the iteration \eqref{iter} and satisfies the pointwise bound \eqref{wi-ineq} (see Remark 1 in \cite{Gallay2010}). The proof is complete. 
\end{proof}
In the next proposition, we estimate the error given by the approximation \eqref{omega-app}, \eqref{Rapp1} and \eqref{Rapp2}.
\begin{proposition}\label{Riapp} There holds
\[
R_{i,app}(\tau,X)\lesssim \kappa e^{-\gamma |X|^2/4}
\]
for any $\gamma\in (0,1)$. 
\end{proposition}
\begin{proof}
We have \[
\bega
R_{i,app}(\tau,X)&=S_i(\tau,X)+\sum_{j\neq i} u_{j,a}\lw(\tau, X+\frac{y^\kappa_{ij}(\tau)}{\sqrt{\kappa e^\tau}}
\rw)\cdot\nabla_X G\\
&\quad+\kappa\eta_0\varphi(\tau)\sum_{j\neq i}u_{j,a}\lw(\tau,X+\frac{y^\kappa_{ij}(\tau)}{\sqrt{\kappa e^\tau}}
\rw)\cdot\nabla_X w_{i,a}+\kappa\eta_0\varphi(\tau)\lw( u_{i,a}\cdot\nabla_X w_{i,a}\rw).
\enda
\]
By using the inequality \eqref{wi-ineq} and the fact that $\|u_{i,a}\|_{L^\infty}\lesssim \|w_{i,a}\|_{L^1\cap L^\infty}\lesssim 1$, we have 
\[\bega 
|R_{i,app}(\tau,X)|&\lesssim \kappa |\pt_\tau w_{i,a}|+|S_i|+\sum_{j,\neq i}|u_{j,a}|_{L^\infty}e^{-\gamma |X|^2/4}+\kappa \|\varphi\|_{L^\infty} \|u_{i,a}\|_{L^\infty} e^{-\gamma |X|^2/4}\\
&\lesssim \kappa e^{-\gamma|X|^2/4}.
\enda 
\]
The proof is complete.
\end{proof}
\subsection{Remainder estimates}
Now we write 
\beq\label{expansionNS}
\begin{cases}
\bar\w_i&=\bar\w_{i,app}+\kappa\eta_0 w_i=G(X)+\kappa \eta_0 \varphi(\tau)w_{i,a}+\kappa\eta_0\eta_0  w_i,\\
\bar u_i&=\bar u_{i,app}+\kappa \eta_0 u_i=v^G(X)+\kappa \eta_0 \varphi(\tau)u_{i,a}+\kappa\eta_0   u_i,
\end{cases}
\eeq
where 
\[
 w_i|_{\tau=0}=0,\qquad u_i=K_B\star_X w_i.
\]
Plugging the above into the equation \eqref{wi-eq}, we get 
\beq\label{wi-main}
\bega
&(\pt_\tau-\mathcal L) w_i+\frac{1}{\kappa\eta_0 }\Lambda w_i+\varphi(\tau)\lw(u_{i,a}\cdot\nabla_X w_i+ u_i\cdot \nabla_X w_{i,a}
\rw)\\
&+\frac{1}{\kappa\eta_0}\sum_{j\neq i} u_j\lw(\tau,X+\frac{y^\kappa_{ij}(\tau)}{\sqrt{\kappa \eta_0 e^\tau}}\rw)\cdot \nabla \bar\w_{i,app}+\frac{1}{\kappa\eta_0}\sum_{j\neq i}\lw\{\bar u_{j,app}\lw(X+\frac{y^\kappa_{ij}(\tau)}{\sqrt{\kappa\eta_0 e^\tau}}
\rw)-u^G\lw(\frac{y^\kappa_{ij}(\tau)}{\sqrt{\kappa\eta_0 e^\tau}}\rw)
\rw\}\cdot\nabla w_i\\
&+\sum_{j\neq i}u_j\lw(\tau,X+\frac{y^\kappa_{ij}(\tau)}{\sqrt{\kappa\eta_0 e^\tau}}
\rw)\cdot \nabla_X w_i+\varphi(\tau)\cdot\frac 1 {\kappa\eta_0} R_{i,app}(\tau,X)+\varphi'(\tau)w_{i,a}=0
\enda 
\eeq
with the initial condition 
\beq \label{initial-wi}
w_i|_{\tau=0}=0.
\eeq
 We recall that 
\[
\wtd R_{i,a}(\tau,X)=\varphi(\tau)R_{i,app}(\tau,X)+\kappa\eta_0 \varphi'(\tau)w_{i,a}
\]
where \[
R_{i,app}(\tau,X)\lesssim \kappa e^{-\gamma |X|^2/4}.
\]
We define the energy norm for the remainder $w_{i,a}$ as follows:
\[
E_w(\tau)=\frac 1 2 \sum_{i=1}^N \|w_i(\tau)\|_{L^2_p}^2,\qquad D_w(\tau)=\sum_{i=1}^N\|\nabla_X w_i(\tau)\|_{L^2_p}^2
\]
where $p(X)=e^{|X|^2/4}$. 
Our goal of this section is to show the apriori estimate, as stated in the following proposition
\begin{proposition}\label{Ew-est}
There holds 
\[
E_w(\tau_\star) \lesssim \int_0^{\tau_\star} \lw(E_w(\tau)+E_w(\tau)^2\rw)d\tau +1+o(1)\cdot \sup_{0\le \tau\le \tau_\star}E_w(\tau).
\]
\end{proposition}~\\
To show the above proposition, we compute, using the equation \eqref{wi-main}:
\[\bega 
E_w'(\tau)&=\la Lw_i w_i \ra_{L^2_p}-\frac 1 {\kappa\eta_0}\la \Lambda w_i,w_i\ra_{L^2_p}-\int_{\R^2}\varphi(\tau)\lw(u_{i,a}\cdot\nabla w_i+ u_i\cdot \nabla w_{i,a}
\rw)w_ip(X)dX\\
&\quad-\frac{1}{\kappa\eta_0} \int_{\R^2} u_j\lw(\tau,X+\frac{y^\kappa_{ij}(\tau)}{\sqrt{\kappa\eta_0 e^\tau}}\rw)\cdot \nabla \bar\w_{i,app}\cdot w_ip(X)dX\\
&\quad-\frac{1}{\kappa\eta_0}\int_{\R^2}\sum_{j\neq i}\lw\{\bar u_{j,app}\lw(X+\frac{y^\kappa_{ij}(\tau)}{\sqrt{\kappa\eta_0 e^\tau}}
\rw)-u^G\lw(\frac{y^\kappa_{ij}(\tau)}{\sqrt{\kappa\eta_0 e^\tau}}\rw)
\rw\}\cdot\nabla w_i\cdot w_i p(X)dX\\
&\quad -\frac{1}{\kappa\eta_0}\int_{\R^2}\varphi(\tau)R_{i,app}(\tau,X)w_ip(X)dX-\varphi'(\tau)\int_{\R^2}w_{i,a}(\tau,X)w_i(\tau,X)p(X)dX.\\
\enda 
\]
In the rest of this section, we estimate each term that appears on the right hand side of the above identity, which will conclude the proof of Proposition \ref{Ew-est}, by the standard Gronwall lemma.
\begin{lemma} There hold
\beq\label{Lwi}
\la \mathcal Lw_i,w_i\ra_{L^2_p}=-\|\nabla_X w_i\|_{L^2_p}^2+\|w_i\|_{L^2_p}^2 
\eeq
and 
\beq\label{Lambda}
\la \Lambda w_i,w_i\ra_{L^2_p}=0.\eeq
\end{lemma}
\begin{proof} We first show \eqref{Lwi}. 
We recall that $\mathcal Lw_i=\triangle w_i+\frac 1 2 X\cdot\nabla_X w_i+w_i$. We have 
\[\bega 
\int_{\R^2}\triangle w_i w_i pdX&=-\int_{\R^2}\nabla w_i \cdot \nabla(pw_i)dX=-\|\nabla w_i\|_{L^2_p}^2+\int_{\R^2}(\nabla w_i \cdot \nabla p)w_i\\
&=-\|\nabla w_i\|_{L^2_p}^2-\frac 1 2 \int (\nabla w_i \cdot X) w_ip.
\enda 
\]
Now we show \eqref{Lambda}. We have 
\[\bega 
\la\Lambda w_i,w_i\ra_{L^2_p}&=\int u^G \cdot\nabla w_i w_i pdX+\int u_i\cdot\nabla G w_i pdX=-\frac 1 {2} \int \div(u^G p)|w_i|^2 -\frac{1}{8\pi}\int (u_i \cdot X)w_i dX.
\enda 
\]
It suffices to show that $\int u_i \cdot X w_i =0$.
By integration by parts, we have 
\[\bega 
\int u_i \cdot X w_i&=\int (u_i^1X_1+u_i^2X_2)(\pt_2 u_i^1-\pt_1 u_i^2)\\
&=\int u_i^2\pt_1(u_i^1X_1+u_i^2X_2)-\int u_i^1\pt_2(u_i^1X_1+u_i^2X_2)\\
&=\int \lw(u_i^2\pt_1 u_i^1 X_1+u_i^2 \pt_1 u_i^2X_2-u_i^1\pt_2 u_i^1X_1-u_i^1\pt_2 u_i^2X_2
\rw)=0.
\enda 
\]
Here, we use the fact that $\pt_1 u_i^1+\pt_2 u_i^2=0$. \end{proof}
\begin{lemma}
There holds 
\[
-\int_{\R^2}\varphi(\tau)\lw(u_{i,a}\cdot\nabla w_i+ u_i\cdot \nabla w_{i,a}
\rw)w_ip(X)dX\lesssim o(1)\|\nabla w_i\|_{L^2_p}^2+\|w_i\|_{L^2_p}^2.
\]
\end{lemma}
\begin{proof}
The inequality follows from the following estimates 
\[\bega
&\int_{\R^2}u_{i,a}\cdot\nabla w_i w_ip(X)dX\lesssim \|u_{i,a}\|_{L^\infty}\|\nabla w_i\|_{L^2_p}\|w_i\|_{L^2_p}\lesssim \|\nabla w_i\|_{L^2_p}\|w_i\|_{L^2_p},\\
&\int_{\R^2}u_i\cdot\nabla w_{i,a}w_ip(X)dX\lesssim \|\nabla w_{i,a}\|_{L^\infty}\|u_i\|_{L^2_p}\|w_i\|_{L^2_p}\lesssim \|w_i\|_{L^2_p}^2.
\enda 
\]
In the above, we used the elliptic estimates in Lemma \ref{elliptic1}. 
The proof is complete.
\end{proof}
\begin{lemma}
There holds 
\[
\frac{1}{\kappa\eta_0} \int_{\R^2} u_j\lw(\tau,X+\frac{y^\kappa_{ij}(\tau)}{\sqrt{\kappa\eta_0 e^\tau}}\rw)\cdot \nabla \bar\w_{i,app}\cdot w_ip(X)dX \lesssim \|w_i\|_{L^2_p}\|w_j\|_{L^2_p}.
\]
\end{lemma}
\begin{proof}
Since $\int_{\R^2}w_idX=0$, using Lemma \ref{elliptic1}, we have 
\[
\sup_{X\in \R^2}(1+|X|^2)|u_j(\tau,X)|\lesssim \|w_j\|_{L^2_p}.
\]
Hence we have 
\[\bega 
&\frac{1}{\kappa} \int_{\R^2} u_j\lw(\tau,X+\frac{y^\kappa_{ij}(\tau)}{\sqrt{\kappa\eta_0 e^\tau}}\rw)\cdot \nabla \bar\w_{i,app}\cdot w_ip(X)dX\\
&\lesssim \frac{1}{\kappa} \int_{\R^2}\frac{\|w_j\|_{L^2_p}}{1+\lw|X+\frac{y_{ij}^\kappa(\tau)}{\sqrt{\kappa\eta_0 e^\tau}}\rw|^2} |\nabla \bar\w_{i,app}||w_i|p(X)dX.\\
\enda 
\]
We divide the above integrals into 2 parts: $|X|\ge \frac 1 2 \frac{|y_{ij}^\kappa(\tau)|}{\sqrt{\kappa\eta_0 e^\tau}}$ and $|X|\le \frac 1 2 \frac{|y_{ij}^\kappa(\tau)|}{\sqrt{\kappa\eta_0 e^\tau}}$. 
In the first case, we have 
\[
|X|\ge \frac 1 2 \frac{|y_{ij}^\kappa(\tau)|}{\sqrt{\kappa \eta_0e^\tau}}\ge \frac{d_T}{2\sqrt{\kappa \eta_0 e^\tau}}.
\]
This implies 
\[\bega 
&\frac{1}{\kappa\eta_0} \int_{|X|\ge \frac{|y_{ij}^\kappa(\tau)|}{2\sqrt{\kappa\eta_0 e^\tau}}}\frac{1}{1+\lw|X+\frac{y_{ij}^\kappa(\tau)}{\sqrt{\kappa e^\tau}}\rw|^2} |\nabla \bar\w_{i,app}||w_i|p(X)dX\\
&\lesssim \|w_i\|_{L^2_p}\cdot \frac 1 \kappa \lw(\int_{|X|\ge \frac{d_T}{2\sqrt{\kappa\eta_0 e^\tau}}}|\nabla_X \bar \w_{i,app}|^2 e^{|X|^2/4}dX\rw)^{1/2}\\
&\lesssim \|w_i\|_{L^2_p} \cdot \frac{1}{\kappa} \lw(\int_{|X|\ge \frac{d_T}{2\sqrt{\kappa\eta_0 e^\tau}}}e^{-\gamma |X|^2/2}e^{|X|^2/4}dX\rw)^{1/2}\lesssim \|w_i\|_{L^2_p}.
\enda 
\]
Now for $|X|\le \frac{|y_{ij}^\kappa(\tau)|}{2\sqrt{\kappa\eta_0 e^\tau}}$, we use the elliptic estimate in Lemma \ref{elliptic1} to get 
\[
\frac 1 \kappa \frac{1}{1+\lw|X+\frac{y_{ij}^\kappa(\tau)}{\sqrt{\kappa\eta_0 e^\tau}}\rw|^2} \le \frac{1}{\kappa}\cdot\frac{1}{\lw|X+\frac{y_{ij}^\kappa(\tau)}{\sqrt{\kappa\eta_0 e^\tau}}\rw|^2}\lesssim  \frac{1}{\kappa} \frac{\kappa e^\tau}{|y_{ij}^\kappa(\tau)|^2}\lesssim 1.
\]
The proof is complete.
\end{proof}
\begin{lemma} There holds
\[
-\frac{1}{\kappa}\int_{\R^2}\sum_{j\neq i}\lw\{\bar u_{j,app}\lw(\tau, X+\frac{y^\kappa_{ij}(\tau)}{\sqrt{\kappa \eta_0e^\tau}}
\rw)-u^G\lw(\frac{y^\kappa_{ij}(\tau)}{\sqrt{\kappa\eta_0 e^\tau}}\rw)
\rw\}\cdot\nabla w_i\cdot w_i p(X)dX\lesssim E_w(\tau).
\]
\end{lemma}
\begin{proof} 
 We have 
\[\bega 
&-\frac{1}{\kappa}\int_{\R^2}\lw\{\bar u_{j,app}\lw(\tau, X+\frac{y^\kappa_{ij}(\tau)}{\sqrt{\kappa\eta_0 e^\tau}}
\rw)-u^G\lw(\frac{y^\kappa_{ij}(\tau)}{\sqrt{\kappa\eta_0 e^\tau}}\rw)
\rw\}\cdot\nabla w_i\cdot w_i p(X)\\
&=\frac 1 {2\kappa}\int_{\R^2}\lw\{\bar u_{j,app}\lw(\tau, X+\frac{y^\kappa_{ij}(\tau)}{\sqrt{\kappa\eta_0 e^\tau}}
\rw)-u^G\lw(\frac{y^\kappa_{ij}(\tau)}{\sqrt{\kappa\eta_0 e^\tau}}\rw)\rw\}\cdot \nabla_X p(X)|w_i|^2dX\\
&\lesssim \int_{\R^2}|w_{j,a}|_{L^\infty}|w_i|^2 p(X)dX+\int_{\R^2}|R_{ij}(\tau,X)||w_i|^2dX\lesssim \|w_i\|_{L^2_p}^2.
\enda 
\]
The proof is complete.
\end{proof}
\begin{lemma}
There holds 
\[
-\frac{1}{\kappa}\int_{\R^2}\varphi(\tau)R_{i,app}(\tau,X)w_ip(X)dX\lesssim E_w(\tau)+1.
\]
and for $\tau_\star>0$:
\[
\int_0^{\tau_\star} |\varphi'(\tau)| \int_{\R^2}\lw|w_{i,a}(\tau,X)w_i(\tau,X)\rw|p(X)dX d\tau\lesssim o(1)\cdot \sup_{0\le\tau\le\tau_\star}E_w(\tau)+1.
\]
\end{lemma}
\begin{proof} The first inequality follows immediately from Proposition \ref{Riapp}. The second inequality follows from Holder inequality and the fact that $\int_0^{\tau_\star}|\varphi'(\tau)|d\tau\lesssim 1$. The proof is complete.
\end{proof}
\begin{lemma} There holds
\[
\int_{\R^2} u_j\lw(\tau,X+\frac{y^\kappa_{ij}(\tau)}{\sqrt{\kappa\eta_0 e^\tau}}
\rw)\cdot \nabla_X w_i w_i p(X)dX\lesssim o(1)\|\nabla_X w_i\|_{L^2_p}^2+E_w(\tau)^2.
\]
\end{lemma} 
\begin{proof}
By Holder inequality, we have 
\[
\int_{\R^2} u_j\lw(\tau,X+\frac{y_{ij}(\tau)}{\sqrt{\kappa\eta_0 e^\tau}}
\rw)\cdot \nabla_X w_i w_i p(X)dX\lesssim \|u_j\|_{L^\infty}\|\nabla_X w_i\|_{L^2_p}\|w_i\|_{L^2_p}\lesssim \|\nabla_X w_i\|_{L^2}E_w(\tau).
\]
The proof is complete.
\end{proof}
Next we obtain the rate of convergence in the following theorem
\begin{theorem}\label{compareNS} Let $d_T>0$ be the positive number satisfying \eqref{gap1}. 
Let \[
d(t,x)=\min_{1\le i\le N}|x-z_i(t)|.\] There exists a constant $C_0>0$ such that 
\[
\lw| u^\kappa(t,x)-\sum_{i=1}^N \al_i K_B(x -z_i(t))\rw| \lesssim \min\lw\{\frac 1 {d(t,x)},\frac{e^{-\frac{d_T^2}{C_0\kappa}}}{d(t,x)^2}
\rw\}+\frac{1}{d(t,x)}e^{-\frac{d(t,x)^2}{16\kappa (t+1)}}+\sqrt\kappa,
\]
as long as 
\[
d(t,x)\ge 2C_0e^{-\frac{d_T^2}{C_0\kappa}}.
\]
\end{theorem}
\begin{proof}
From the asymptotic expansion \eqref{expansionNS}, we have 
\[
\lw |u^\kappa(t,x)-\sum_{i=1}^N \al_i K_B(x-z_i^\kappa(t))\lw(1-e^{-\frac{|x-z_i^\kappa(t)|^2}{4\kappa(t+1)}}
\rw)\rw|\lesssim  \sqrt{\kappa} .
\] 
Therefore, it suffices to show that 
\[
K_B(x-z_i^\kappa(t))\lw(1-e^{-\frac{|x-z_i^\kappa(t)|^2}{4\kappa(t+1)}}
\rw)-K_B(x-z_i(t)).
\]
Let $\wtd x=x-z_i(t)$ and $\wtd h=z_i(t)-z_i^\kappa(t)$. We have 
\[\bega 
K_B(\wtd x+\wtd h)\lw(1-e^{-\frac{|\wtd x+\wtd h|^2}{4\kappa(t+1)}}
\rw)-K_B(\wtd x)&=K_B(\wtd x+\wtd h)-K_B(\wtd x)-K_B(\wtd x+\wtd h)e^{-\frac{|\wtd x+\wtd h|^2}{4\kappa(t+1)}}\\
&\lesssim \frac{|\wtd h|}{|\wtd x+\wtd h|^2}\lw(1+\frac{|\wtd h|}{|\wtd x|}
\rw)+\frac{1}{|\wtd x+\wtd h|}e^{-\frac{|\wtd x+\wtd h|^2}{4\kappa(t+1)}}\\
&\lesssim \frac{|\wtd h|}{|\wtd x|^2}+\frac{1}{|\wtd x|}e^{-\frac{|\wtd x|^2}{16\kappa(t+1)}}\\
&\lesssim \min\lw\{ \frac{1}{|\wtd x|} ,\frac{ e^{-\frac{d_T^2}{C_0\kappa}}}{|\wtd x|^2}\rw\}+\frac{1}{|\wtd x|}e^{-\frac{|\wtd x|^2}{16\kappa(t+1)}}.\\
\enda \]
as long as $|\wtd x|\ge 2|\wtd h|$. Here, we used Lemma \ref{gap-lem} to get $
|\wtd h|\le C_0 e^{-\frac{d_T^2}{C_0\kappa}}
$
for some $C_0>0$. Hence as long as $|\wtd x|\ge 2C_0 e^{-\frac{d_T^2}{C_0\kappa}}$, we have the desired inequality.\end{proof}
\subsection{Higher derivative estimates} \label{higher-sec}
In this section, we give uniform estimates on higher derivatives of the Navier-Stokes solution $u^\kappa$. Our main theorem is as follows
\begin{theorem} \label{deri-thm} Let $m,n\ge 0$ be an integer. 
There exists $C_0>0$ such that 
\[
\|\pt_t^n u^\kappa\|_{W^{m,\infty}}+\|\pt_t^n \nabla_x u^\kappa\|_{H^m}+\|\pt_t^n p^\kappa\|_{H^m}+\|\pt_t^n p^\kappa\|_{W^{m,\infty}} \lesssim \exp\lw(C_0e^{\kappa^{-1}}
\rw).
\]
\end{theorem}
\begin{proof} First we consider $n=0$. We define 
\[
u^\kappa_a(t,x)=\sum_{i=1}^N  \frac{\al_i}{\sqrt{\kappa\eta_0(t+1)}}u^G\lw(\frac{x-z_i^\kappa(t)}{\sqrt{\kappa\eta_0(t+1)}}
\rw).
\]
It is straightforward that $u^\kappa(t)-u_a^\kappa(t)\in L^2(\R^2)$. Moreover, we have $u_0^\kappa(x)=u^\kappa_a|_{t=0}$. 
At the same time, we define the approximate pressure
\[
p_a^\kappa(t,x)=\sum_{i=1}^N \frac{\al_i}{\sqrt{\kappa\eta_0(t+1)}} p^G\lw(\frac{x-z_i^\kappa(t)}{\sqrt{\kappa\eta_0(t+1)}}
\rw)
\]
where $p^G=p^G(X)$ is determined through the elliptic equation 
\[
-\triangle_X p^G=\div(v^G\cdot\nabla_X v^G).
\]
Since a single point vortex is a solution to the Navier-Stokes equations, we have 
\beq\label{vG-eq}
\pt_t v^G+v^G\cdot\nabla_X v^G+\nabla_X p^G-\kappa\eta_0\triangle_Xv^G=0
\eeq
Now we write 
\[
u^\kappa(t,x)=u^\kappa_a(t,x)+u^\kappa_R(t,x),\qquad p^\kappa(t,x)=p_a^\kappa(t,x)+p_R^\kappa(t,x)
\]
and $u^\kappa_R$ solves the equation 
\beq\label{uR-eq}
\begin{cases}
&\pt_t u^\kappa_R+u_R^\kappa\cdot \nabla u_a^\kappa+u_a^\kappa\cdot \nabla u_R^\kappa+u_R^\kappa\cdot\nabla u_R^\kappa+\nabla p^\kappa_R=\kappa\eta_0 \triangle u^\kappa_R+\mathcal F_{R},\\
&\div(u_R^\kappa)=0.
\end{cases}
\eeq
The approximation error $\mathcal F_R$ is computed from \eqref{vG-eq} as 
\[\bega 
\mathcal F_R&=\pt_t u_a^\kappa+u^\kappa_a\cdot\nabla_x u_a^\kappa+\nabla p_a^\kappa-\kappa\eta_0\triangle u^\kappa_a\\
&=\frac{1}{(\kappa\eta_0(t+1))^{\frac 3 2}}\sum_{j\neq i} \al_j\alpha_i\lw\{u^G\lw(\frac{x-z_j^\kappa(t)}{\sqrt{\kappa\eta_0(t+1)}}\rw)\cdot\nabla_X u^G\lw(\frac{x-z_i^\kappa(t)}{\sqrt{\kappa\eta_0(t+1)}}
\rw)\rw\}.
\enda 
\]
Let $m\ge 0$ be an integer. Using the fact that $\|u^G\|_{L^\infty_X}+\|\nabla_X u^G\|_{H^m}\lesssim 1$, we have 
\beq\label{FRbound}
\|\mathcal F_R\|_{H^m}\lesssim \kappa^{-\frac{m+3}{2}}.
\eeq
Performing a standard energy estimate for the equation \eqref{uR-eq}, we have 
\[
\frac{d}{dt}\| u^\kappa_R\|_{H^m}^2\lesssim \|u_a^\kappa\|_{W^{m+1,\infty}}\|u_R^\kappa\|_{H^m}^2+\|\nabla u_\kappa^R\|_{L^\infty}\|u_R^\kappa\|_{H^m}^2+\|\mathcal F_R\|_{H^m}\|u_R^\kappa\|_{H^m}.
\]
Using the standard elliptic estimate 
\[
\|\nabla_x u^\kappa_R\|_{L^\infty}\lesssim (1+\|\w^\kappa_R\|_{L^\infty})\lw(1+\ln^+(\|u^\kappa_R\|_{H^3})+\ln^+(\|\w^\kappa_R\|_{L^2})
\rw),
\]
we get 
\[
\frac{d}{dt}\|u_R^\kappa\|_{H^m}\lesssim \|u_a^\kappa\|_{W^{m+1,\infty}}\|u_R^\kappa\|_{H^m}+\|u_R^\kappa\|_{H^m}\lw(\ln^+\|u_R^\kappa\|_{H^m}+1+\|\w_R^\kappa\|_{L^\infty}\ln^+\|\w_R^\kappa\|_{L^2}
\rw)+\| \mathcal F_R\|_{H^m}.
\]
By  $\|\w_R^\kappa\|_{L^\infty}\lesssim \kappa^{-1}$, $\|u_a^\kappa\|_{W^{m+1,\infty}}\lesssim \kappa^{-m/2-1}$ and the inequality \eqref{FRbound} we have 
\[
\frac{d}{dt}\|u_R^\kappa\|_{H^m}\lesssim \kappa^{-\frac{m}{2}-1}\|u_R^\kappa\|_{H^m}+\kappa^{-1}\|u_R^\kappa\|_{H^m}\ln(1+\|u_R^\kappa\|_{H^m})+\kappa^{-\frac{m+3}{2}}.\]
By the Gronwall inequality, we have 
\[
\|u_R^\kappa(t)\|_{H^m}\lesssim \exp\lw(C_0e^{\kappa^{-1}}
\rw)
\]
for some constant $C_0>0$. Finally, we have
\[
\|\nabla_x u^\kappa\|_{H^m}\le \|\nabla_x u_a^\kappa\|_{H^m}+\|\nabla u_R^\kappa\|_{H^m}\lesssim \exp(C_0e^{\kappa^{-1}}).
\]
The bound on $\|u^\kappa\|_{L^\infty}$ follows from Sobolev embedding with $m$ is chosen to be large. Now we estimate the pressure. From the Navier-Stokes equation, we have 
\[-\triangle p^\kappa=\sum_{i,j\le 2} \pt_i \pt_j(u^\kappa_i u^\kappa_j).
\]
 Taking derivatives in $x$ and $t$ of both sides, we have
\[
-\triangle(\pt_x^m \pt_t^np^\kappa)=\sum_{i,j\le 2}\pt_i\pt_j(\pt_x^m\pt_t^n(u_i^\kappa u_j^\kappa)).
\]
This implies 
\[
\pt_t^n\pt_x^m p^\kappa(t,x)=\sum_{i,j} \mathcal R_i \mathcal R_j (\pt_t^n\pt_x^m (u_i^\kappa u_j^\kappa)).
\]
where $\mathcal R_i$ is the Riesz transform. Using the fact that the operator $\mathcal R_i$ is bounded from $L^2$ to $L^2$, we have 
\beq\label{pk1}
\|\pt_t^n p^\kappa\|_{H^m}\lesssim \sum_{i,j}\|\pt_t^n (u_i^\kappa u_j^\kappa)\|_{H^m}.
\eeq
When $n=0$, if $m\ge 1$, the estimate for $\|p^\kappa\|_{H^m}$ follows from 
\[
\|p^\kappa\|_{H^m}\lesssim \|u^\kappa\|_{W^{m,\infty}}\|\nabla u^\kappa\|_{H^m}\lesssim \exp\lw(2C_0 e^{\kappa^{-1}}\rw),
\]
while if $m=0$, we have $\|p^\kappa\|_{L^2}\lesssim \||u_\kappa|^2\|_{L^2}\lesssim \exp(C_0e^{\kappa^{-1}})$, using the expansion \eqref{expansionNS} and the fact that $u^G(X)$ decays like $\frac{1}{|X|}$ as $|X|\to\infty$. ~\\
Now for $\|\pt_t^n u^\kappa\|_{H^m}$ and $\|\pt_t^n p^\kappa\|_{H^m}$, we use induction on $n$. Namely, for $n=1$, we have 
\[
\pt_t u^\kappa=-u^\kappa\cdot\nabla u^\kappa-\nabla p^\kappa+\kappa\eta_0\triangle u^\kappa.
\]
This implies 
\[
\|\pt_t u^\kappa\|_{H^m}\lesssim \|u^\kappa\|_{W^{m,\infty}}\|\nabla_x u^\kappa\|_{H^{m}}+\|p^\kappa\|_{H^{m+1}}+\kappa\|\nabla_x u^\kappa\|_{H^{m+1}}\lesssim \exp\lw(2C_0e^{-\kappa^{-1}}\rw).
\]
From the above, we can estimate $\|\pt_t p^\kappa\|_{H^m}$ by using the inequality \eqref{pk1} and the bound on $\|\pt_t u^\kappa\|_{H^m}$ above. We shall skip the details for the remaining argument. The proof is complete.
\end{proof}
In other to simplify the notations for the Navier-Stokes parts, we denote $c_\kappa$ to be 
\[
c_\kappa=\exp(C_0e^{\kappa^{-1}}).
\]
\begin{remark}\label{c-remark}
By possibly increasing the value of the constant $C_0$, we may replace all the upper bounds involving $c_\kappa^2,c_\kappa^3,\cdots$ or $\kappa^{-M}$ by just $c_\kappa$ for simplicity of notation. 
\end{remark}

\section{Construction of the Boltzmann solutions}\label{BE-con}
In this section, we construct solution of the Boltzmann equation \eqref{F-eq} around the vortex solutions. 
We define 
\[
\Gamma(F,G)=\frac 1 {\sqrt\mu} Q\lw(\mu^{\frac1 2}F,\mu^{\frac 1 2}G\rw)=\Gamma_{\text{gain}}(F,G)-\Gamma_{\text{loss}}(F,G),
\]
where 
\[\bega 
\Gamma_{\text{gain}}(F,G)&=\int_{\mathbb R^3\times \mathbb S^2}|(v-v_\star)\cdot \wtd s|\sqrt{\mu(v_\star)}\lw(F(v')G(v_\star')+F(v_\star')G(v')
\rw)d\wtd s dv_\star,\\
\Gamma_{\text{loss}}(F,G)&=\int_{\mathbb R^3\times \mathbb S^2}|(v-v_\star)\cdot \wtd s|\sqrt{\mu(v_\star)}\lw(F(v)G(v_\star)+F(v_\star)G(v)
\rw)d\wtd s dv_\star.\\
\enda 
\]
We recall the linear operator $Lf$ to be 
\beq\label{L-def}
Lf=-\frac{2}{\sqrt\mu}Q(\mu,\sqrt\mu f)=\nu(v)f(v)-\mathcal Kf(v)
\eeq
where 
\[\begin{cases}
\nu(v)&=\int_{\R^3\times\mathbb S^2}|(v-v_\star)\cdot\wtd s|\mu(v_\star)d\wtd s dv_\star\sim (1+|v|^2)^{1/2},\\
\mathcal Kf(v)&=\int_{\R^3}K(v,v_\star)f(v_\star)dv_\star,\\
K(v,v_\star)&=C_1|v-v_\star|e^{-\frac{|v|^2+|v_\star|^2}{4}}-\frac{C_2}{|v-v_\star|}
e^{-\frac{|v-v_\star|^2}{8}-\frac 1 8 \cdot \frac{(|v|^2-|v_\star|^2)^2}{|v-v_\star|^2}}.\\
\end{cases}
\]
We recall the Grad estimate (\cite{GlasseyBook}) :
\beq\label{grad}
K(v,v_\star)\lesssim K_\vartheta(v,v_\star)\quad\text{for}\quad \vartheta\in\lw(0,\frac {1}{8}\rw),\qquad K_\vartheta(v,v_\star)=\frac{1}{|v-v_\star|}e^{-\vartheta|v-v_\star|^2-\vartheta \frac{(|v|^2-|v_\star|^2)^2}{|v-v_\star|^2}}.
\eeq
Let $\mathcal N$ be the nullspace of $L$ in $L^2_v(\R^3)$, then 
\[
\mathcal N=\text{span}\lw\{\sqrt\mu,v_1\sqrt\mu,v_2\sqrt\mu,v_3\sqrt\mu,\frac{|v|^2-3}{2}\sqrt\mu
\rw\}
\]
We also denote $\mathcal N^\perp$ to be the subspace of $L^2_v(\R^3)$ that is perpendicular to $\mathcal N$ in the Hilbert space $L^2_v(\R^3)$. For a function $f(t,x,v)\in L^2(\R_v^3)$, we define $Pf$ to be the projection of $f$ onto the space $\mathcal N$, and $(I-P)f$ the projection of $f$ onto the space $\mathcal N^\perp$.
We recall the coercivity property if $L$ in the space $\mathcal N^\perp$:
\beq\label{Lff}
\la Lf,f\ra_{L^2(\R^3_v)}\ge c_0\|\sqrt\nu f\|_{L^2(\R^3_v)}^2\qquad \text{for all}\qquad f\in \mathcal N^\perp.
\eeq
We define \beq\label{Aij}
\hat A_{ij}(v)=\lw(v_iv_j-\frac 1 3 \delta_{ij}|v|^2\rw)\sqrt\mu,\qquad A_{ij}(v)=L^{-1}\lw(\lw(v_iv_j-\frac 1 3 \delta_{ij}|v|^2\rw)\sqrt\mu
\rw)
\eeq
and 
\beq\label{Bj}
\hat B_j =v_j \frac{|v|^2-5}{2}\sqrt\mu,\qquad B_j=L^{-1}B_j.
\eeq
It is standard (see, for example \cite{CaoJangKim}) that 
\beq\label{AijAkl}
\la LA_{ij},A_{kl}\ra=\eta_0\lw(\delta_{ik}\delta_{jl}+\delta_{il}\delta_{jk}-\frac 2 3 \delta_{ij}\delta_{kl}
\rw)
\eeq
and 
\beq \label{BjBk}
\la \hat B_j, B_k\ra=\eta_c \delta_{jk}.
\eeq
for some universal constants $\eta_0>0$ and $\eta_c>0$.
\subsection{Hilbert expansion}\label{H-exp}
Our goal is to construct solution, using the expansion 
\beq\label{BEexp}
F=\mu+\eps\sqrt{\mu}f_1+\eps^2 \sqrt\mu f_2+\eps^3\sqrt\mu f_3+\eps \delta \sqrt\mu f_R=\mu+\sqrt\mu f_a+\eps\delta\sqrt\mu f_R
\eeq
where $f_a$ is the approximate solution, given by 
\[
f_a(t,x,v)=\eps f_1+\eps^2 f_2+\eps^3 f_3.
\]
Plugging \eqref{BEexp} into the equation \eqref{F-eq}, we have the equation for the remainder $f_R$ as follows 
\[\bega 
\pt_t f_R+\frac 1 \eps v\cdot\nabla_x f_R+\frac{1}{\kappa\eps^2}Lf_R&=\frac{2}{\kappa\eps^2}\Gamma(f_a,f_R)+\frac{\delta}{\kappa\eps}\Gamma(f_R,f_R)+R_{B,a}(f_a)\\
\enda
\]
where 
\beq\label{BE-error}
R_{B,a}(f_a)=-\frac{1}{\eps\delta}\lw\{\pt_t f_a+\frac 1 \eps v\cdot\nabla_x f_a+\frac{1}{\kappa\eps^2}Lf_a-\frac{1}{\kappa\eps^2}\Gamma(f_a,f_a)
\rw\}.
\eeq
Using the expansion $f_a=\eps f_1+\eps^2 f_2+\eps^3 f_3$, we have 
\beq\label{RB-long}
\bega 
R_{B,a}(f_a)&=-\frac{1}{\eps\delta}\lw\{\frac 1 {\kappa\eps}Lf_1\rw\}\\
&\quad-\frac{1}{\eps\delta}\lw\{v\cdot\nabla_x f_1 +\frac{1}{\kappa}Lf_2-\frac 1 \kappa \Gamma(f_1,f_1)
\rw\}\\
&\quad-\frac {1}{\eps\delta}\cdot \eps \lw\{\pt_t f_1+v\cdot\nabla_x f_2+\frac 1 \kappa Lf_3-\frac 2 \kappa \Gamma(f_1,f_2)
\rw\}\\
&\quad-\frac {1}{\eps\delta}\cdot \eps^2 \lw\{\pt_t f_2+v\cdot\nabla_x f_3-\frac{1}{\kappa}\lw(\Gamma(f_2,f_2)+2\Gamma(f_1,f_3)
\rw)
\rw\}\\
&\quad-\frac 1 {\eps\delta}\cdot \eps^3 \lw\{\pt_t f_3-\frac 2 \kappa \Gamma(f_2,f_3)-\frac{\eps}{\kappa}\Gamma(f_3,f_3)
\rw\}.
\enda 
\eeq
For $1\le i\le 3$, we denote 
\[
Pf_i=\lw(a_i(t,x)+\sum_{k=1}^3 b^k_i(t,x) v_k+c_i(t,x)\frac{|v|^2-3}{2}\rw)\sqrt\mu 
\]
\begin{proposition}\label{f1f2}
Let $u^\kappa$ be the solution to incompressible the Navier-Stokes equation \eqref{NSeq}.  Define 
\beq\label{f1}
f_1(t,x,v)=\lw\{u^\kappa(t,x)\cdot v\rw\}\sqrt\mu
\eeq
and 
\beq\label{micro-f2}
(I-P)f_2(t,x,v)=\frac 1 2 \sum_{i,j}u_i^\kappa u_j^\kappa \hat A_{ij}(v)-\kappa\sum_{i,j}\pt_j u_i^\kappa A_{ij}(v).
\eeq
Let $b_2$ be the solution to the forced incompressible linearized Navier-Stokes equations around $u^\kappa$
\beq\label{b2}
\bega 
&\pt_t b_2+\eta_0(u^\kappa\cdot\nabla_x b_2+b_2\cdot\nabla_x u^\kappa)+\nabla_x p_2\\
&=\kappa\eta_0\triangle b_2-\sum_{j}\la A_{ij},\pt_j (2\Gamma(f_1,(I-P)f_2)-\kappa (v\cdot\nabla_x ((I-P)f_2)))\ra_{L^2_v}\\
\enda \eeq
with the divergence free condition  $\div(b_2)=0$ and zero datum $b_2|_{t=0}=0$.
Let $c_2$ solve the convected heat equation 
\beq\label{c2}
\bega
&2\pt_t c_2+\eta_c u^\kappa\cdot\nabla c_2-\kappa\eta_c\triangle c_2 \\
&=-\sum_j \lw\la B_j,\pt_j(2\Gamma(f_1,(I-P)f_2)-\kappa (v\cdot\nabla_x (I-P)f_2))\rw\ra-\pt_t \lw(\frac 1 3 |u^\kappa|^2+p^\kappa
\rw).
\enda 
\eeq
with zero initial data. Define \beq\label{a2}
a_2=\frac 1 3 |u^\kappa|^2+p^\kappa-c_2.
\eeq
Let $b_3$ solves the elliptic problem 
\beq\label{b3}
\div(b_3)=-\pt_t a_2.
\eeq
Define 
\beq \label{f2f3}
\bega 
Pf_2&=\lw\{a_2(t,x)+b_2(t,x)\cdot v+c_2(t,x)\frac{|v|^2-3}{2}\rw\}\sqrt\mu,\\
Pf_3&=\lw(p_2+\frac 2 3 u^\kappa\cdot b_2\rw)\frac{|v|^2-3}{2}\sqrt\mu,\\
(I-P)f_3&=2L^{-1}\lw\{\Gamma(f_1,f_2)\rw\}-\kappa L^{-1}\lw\{\pt_t f_1+v\cdot\nabla_x f_2\rw\}.
\enda 
\eeq
Here $p_2$ is the pressure coming from the equation \eqref{b2}. 
Then $R_{B,a}(f_a)$ in \eqref{BE-error} satisfies 
\[\bega 
R_{B,a}(f_a)&=-\frac{\eps}{\delta}(I-P)\lw\{\pt_t f_2+v\cdot\nabla_x f_3-\frac 1 \kappa \Gamma(f_2,f_2)-\frac 2 \kappa \Gamma(f_1,f_3)
\rw\}\\
&\quad-\frac{\eps^2}{\delta}\lw\{\pt_t f_3-\frac 2 \kappa \Gamma(f_2,f_3)-\frac \eps\kappa \Gamma(f_3,f_3)
\rw\}.
\enda 
\]
\end{proposition}
\begin{proof}
Since $f_1=u^\kappa\cdot v\sqrt\mu \in \mathcal N$, we have $Lf_1=0$. The second line in \eqref{RB-long} gives us 
\beq\label{f2-eq}
(I-P)f_2=L^{-1}\lw(\Gamma(f_1,f_1)-\kappa v\cdot\nabla_x f_1
\rw)
\eeq 
Using $L^{-1}(\Gamma(f_1,f_1))=(I-P)\lw(\frac{f_1^2}{2\sqrt\mu}
\rw)$ for $f_1\in \mathcal N$, we have 
\[\bega 
(I-P)f_2&=(I-P)\lw\{\frac 1 2|u^\kappa\cdot v|^2\sqrt\mu
\rw\}-\kappa L^{-1}(I-P)(v_j\pt_j u_i^\kappa v_i\sqrt\mu)\\
&=\frac 1 2 (I-P)\lw\{u^\kappa_i u^\kappa_j v_i v_j \sqrt\mu\rw\} -\kappa L^{-1}(I-P)(\pt_j u_i^\kappa v_i v_j\sqrt\mu)\\
&=\frac 1 2 u_i^\kappa u_j^\kappa \hat A_{ij}-\kappa\pt_j u_i^\kappa A_{ij}.
\enda 
\]
The proof of \eqref{micro-f2} is complete. Moreover, $P(v\cdot\nabla_x f_1)=0$ since $u^\kappa$ is divergence free. Therefore, the term on the second line of \eqref{RB-long} is zero. Now we show that $f_2,f_3$ defined above will solve the equation \beq\label{cancel1}
\pt_t f_1+v\cdot\nabla_x f_2+\frac 1 \kappa Lf_3-\frac 2 \kappa \Gamma(f_1,f_2)=0,
\eeq
which is the term on the third line in \eqref{RB-long}. By the definition of $(I-P)f_3$, it suffices to check that 
$P(\pt_t f_1+v\cdot\nabla_x f_2)=0$, which is equivalent to 
\[
\int_{\R^3} (\pt_t f_1+v\cdot\nabla_x f_2)\bmx
\sqrt\mu\\ v_i\sqrt\mu\\\frac{|v|^2-3}{2}\sqrt\mu\\
\emx dv=0.
\]
Using Lemma \ref{v-dot-Pf} and the facts that $\div(b_2)=0$, $f_1=u^\kappa\cdot v\sqrt\mu$, the above is equivalent to 
\beq \label{check1}
\begin{cases}
\la  v\cdot\nabla_x (I-P)f_2,\sqrt\mu\ra_{L^2_v}&=0,\\
\pt_t u^\kappa +\nabla_x  (a_2+c_2)+\la v\cdot\nabla_x (I-P)f_2,v\sqrt\mu\ra_{L^2_v}&=0,\\
\pt_t c_1+\la  v\cdot\nabla_x (I-P)f_2,\frac{|v|^2-3}{2}\sqrt\mu\ra_{L^2_v}&=0.\\
\end{cases}
\eeq
The first equation follows directly from Lemma \ref{v-dot-L1}. Now we check in the second equation in the above. Using the equation \eqref{f2-eq}, Lemma \ref{convect} and Lemma \ref{v-dot-L1}, it suffices to check that 
\[
\pt_t u^\kappa_i +\pt_i(a_2+c_2)+u^\kappa\cdot\nabla_x u^\kappa_i -\frac 1 3\pt_i(|u^\kappa|^2)-\kappa \sum_j \la A_{ij},\pt_{x_j}(v\cdot \nabla_x f_1)\ra_{L^2_v}=0.
\]
The above follows directly from the fact that $u^\kappa$ solves the Navier-Stokes equations, the definitions of $a_2$ in \eqref{f2f3}, and the calculations in Lemma \ref{fg-lem}. Now we check that last equation in \eqref{check1}. Using the equation \eqref{f2-eq} and Lemma \ref{v-dot-L1}, it suffices to check that 
\[
\pt_t c_1+\sum_j \la B_j,\pt_j (I-P)f_2\ra_{L^2_v}=0.
\]
The above equation follows directly from Lemma \ref{fg-lem} and the fact that $c_1=0$ by the equation \eqref{f1}. The proof of \eqref{cancel1} is complete. Now we proceed with the term on the forth line in \eqref{RB-long}. We will show that 
\[
P(\pt_t f_2+v\cdot\nabla_x f_3)=0.
\]
In other words,
\[
\int_{\R^3}\lw(\pt_t f_2+v\cdot\nabla_x f_3\rw)\bmx \sqrt\mu\\v_i\sqrt\mu\\\frac{|v|^2-3}{2}\sqrt\mu
\emx dv=0.
\]
This is equivalent to 
\[
\pt_t \bmx 
a_2\\b_2\\c_2\\
\emx +\int v\cdot\nabla_x \lw\{Pf_3+L^{-1}\lw(2\Gamma(f_1,f_2)-\kappa (\pt_t f_1+v\cdot\nabla_x f_2)\rw)\rw\}\bmx
\sqrt\mu\\ v_i\sqrt\mu\\\frac{|v|^2-3}{2}\sqrt\mu\\
\emx dv=0.
\]
Using Lemma \ref{v-dot-Pf}, Lemma \ref{v-dot-L1} for the second term, it suffices to check that 
\beq\label{random23}
\begin{cases} 
&\pt_ta_2+\div(b_3)=0,\\
&\pt_t b_2+\nabla_x (a_3+c_3)+\sum_j \la A_{ij},\pt_j (2\Gamma(f_1,f_2)-\kappa (\pt_t f_1+v\cdot\nabla_x f_2))\ra_{L^2_v}=0,\\
&\pt_t c_2+\div(b_3)+\sum_j \la B_j,\pt_j(2\Gamma(f_1,f_2)-\kappa (\pt_t f_1+v\cdot\nabla_x f_2)) \ra_{L^2_v}=0.
\end{cases}
\eeq
The first equation in the above is true, by the definition of $b_3$ in \eqref{b3}. We now verify the second equation. To this end, we rewrite it as
\[\bega 
&\pt_t b_2+\nabla_x (a_3+c_3)+\sum_j \la A_{ij},\pt_j \lw\{2\Gamma(f_1,Pf_2)-\kappa (\pt_t f_1+v\cdot\nabla_x Pf_2)\rw\}\ra_{L^2_v}\\
&=-\sum_j \la A_{ij},\pt_j (2\Gamma(f_1,(I-P)f_2)-\kappa (v\cdot\nabla_x ((I-P)f_2)))\ra_{L^2_v}.
\enda 
\]
This follows directly, thanks to Lemma \eqref{fg-lem}, the definition \eqref{b2} and \eqref{f2f3}. The last equation \eqref{random23} follows directly from Lemma \eqref{fg-lem} and the equations \eqref{c2}. The proof is complete.
\end{proof}
\begin{proposition}\label{deri-thm2}
Let $c_\kappa=\exp(C_0e^{\frac 1 \kappa})$ be the constant so that the estimates in Theorem \ref{deri-thm} holds. One can take a larger constant $C_0$ such that for $2\le i\le 3$, there holds  
\[
\|\pt_t^n (a_i,b_i,c_i)\|_{W^{m,\infty}}+\|\pt_t^n \nabla_x (a_i,b_i,c_i)\|_{H^m}+\|\pt_t^n p_2\|_{H^m}+\|\pt_t^n p_2\|_{W^{m,\infty}} \lesssim  e^{c_\kappa}
\]
We also define $\bar c_\kappa=e^{c_\kappa}$.
\end{proposition}
\begin{proof}
We first bound $c_2$, which solves the equation \eqref{c2}. The equation can be written as 
\[
2\pt_t c_2+\eta_c u^\kappa\cdot\nabla c_2-\kappa\eta_c\triangle c_2=f_{c_2}
\]
where 
\[
f_{c_2}=-\sum_j \lw\la B_j,\pt_j(2\Gamma(f_1,(I-P)f_2)-\kappa (v\cdot\nabla_x (I-P)f_2))\rw\ra-\pt_t \lw(\frac 1 3 |u^\kappa|^2+p^\kappa
\rw).
\]
Here we note that $(I-P)f_2$ is given \eqref{f2-eq} and $f_1=u^\kappa\cdot v\sqrt\mu$, meaning the forcing term $f_{c_2}$ only depends on $u^\kappa$, which is the solution of the Navier-Stokes equations. By a standard energy estimate, we have, for $m\ge 0$: 
\[\bega 
\frac d {dt} \|c_2\|_{H^m}^2&\lesssim \|u^\kappa\|_{W^{m-1,\infty}}\|c_2\|_{H^m}^2+\|f_{c_2}\|_{H^m}\|c_2\|_{H^m},\\
\frac{d}{dt} \|c_2\|_{H^m}&\lesssim c_\kappa \|c_2\|_{H^m}+\|f_{c_2}\|_{H^m}.
\enda 
\]
This implies 
\[
\|c_2(t)\|_{H^m}\lesssim e^{C_0c_\kappa t}\int_0^t \|f_{c_2}(s)\|_{H^m}ds.
\]
From Theorem \ref{deri-thm}, it is straightforward that $\|f_{c_2}(s)\|_{H^m}\lesssim c_\kappa$. Hence we have 
\[
\|c_2(t)\|_{H^m}\lesssim e^{c_\kappa}
\]
by possibly increasing the value of the constant $C_0$. The bounds for the time derivatives follow directly, by using the fact that $2\pt_tc_2=-\eta_c u^\kappa\cdot\nabla c_2+\kappa\eta_c\triangle c_2+f_{c_2}$. The bounds for $a_2$ follows directly from Theorem \ref{deri-thm} and the equation \eqref{a2}. Now we bound $b_3$, which solves the equation \eqref{b3}. We have 
\[
\pt_1 b_3^1+\pt_2 b_3^2=-\pt_ta_2.
\]
We look for $b_3$ of the form $b_3=(\pt_1\psi_{b_3},\pt_2\psi_{b_3})$. The above equation becomes
\[
\triangle \psi_{b_3}=-\pt_ta_2.
\]
A standard elliptic estimate yields 
\[
\|b_3\|_{H^m}\lesssim \|\nabla \psi_{b_3}\|_{H^m}\lesssim \|\pt_ta_2\|_{H^m}\lesssim e^{c_\kappa}.
\]
Next, we bound $b_2$, which solves the linearized Navier-Stokes equations \eqref{b2}.
The equation \eqref{b2} can be written as 
\[
\pt_t b_2+\eta_0(u^\kappa\cdot\nabla_x b_2+b_2\cdot\nabla_x u^\kappa)+\nabla_x p_2-\kappa\eta_0\triangle b_2=f_{b_2}.
\]
where \[
f_{b_2}=-\sum_{j}\la A_{ij},\pt_j (2\Gamma(f_1,(I-P)f_2)-\kappa (v\cdot\nabla_x ((I-P)f_2)))\ra_{L^2_v}
\]
The proof for $b_2$ runs exactly as in Theorem \eqref{deri-thm} (see the equation \eqref{uR-eq}. We skip the details. The proof is complete.
\end{proof}
\begin{remark} \label{c-bar-remark} As in Remark \ref{c-remark}, by possibly increasing the value of $C_0$, we can replace all the upper bounds involving $\bar c_\kappa^2,\bar c_\kappa^3,\cdots$ by $\bar c_\kappa$.
\end{remark}
Plugging the expansion \eqref{BEexp} into the equation \eqref{F-eq}, and using the definitions of $f_1,f_2,f_3$ in the above proposition, we obtain the equation for the remainder $f_R$ as follows
\beq\label{fR-eq}
\bega 
\lw(\pt_t +\frac 1 \eps v\cdot\nabla_x\rw)f_R+\frac{1}{\kappa\eps^2}L(I-P)f_R&=\frac{2}{\kappa\eps}\Gamma(f_1+\eps f_2+\eps^2 f_3,f_R)+\frac{\delta}{\kappa\eps}\Gamma(f_R,f_R)\\
&\quad-\frac{\eps}{\delta}(I-P)\lw\{\pt_t f_2+v\cdot\nabla_x f_3-\frac 1 \kappa \Gamma(f_2,f_2)-\frac 2 \kappa \Gamma(f_1,f_3)
\rw\}\\
&\quad-\frac{\eps^2}{\delta}\lw\{\pt_t f_3-\frac 2 \kappa \Gamma(f_2,f_3)-\frac \eps\kappa \Gamma(f_3,f_3)
\rw\}
\enda 
\eeq
The above equation can be written as 
\beq\label{fR-eq1}
\pt_tf_R +\frac 1 \eps v\cdot\nabla_x f_R+\frac {1}{\kappa\eps^2}Lf=g
\eeq
where 
\[
g=\sum_{m=1}^4 g_m
\]
and 
\beq\label{g123}
\begin{cases}
g_1(t,x,v)&=-\frac{\eps}{\delta}(I-P)\lw\{\pt_t f_2+v\cdot\nabla_x f_3-\frac 1 \kappa \Gamma(f_2,f_2)-\frac 2 \kappa \Gamma(f_1,f_3)
\rw\}\\
&\quad+\frac{2\eps^2}{\kappa\delta}\Gamma(f_2,f_3)+\frac{\eps^3}{\kappa\delta}\Gamma(f_3,f_3),\\
g_2(t,x,v)&=\frac{2}{\kappa\eps}\Gamma(f_1+\eps f_2+\eps^2 f_3,f_R),\\
g_3(t,x,v)&=-\frac{\eps^2}{\delta}\pt_t f_3, \\
g_4(t,x,v)&=\frac{\delta}{\kappa\eps}\Gamma(f_R,f_R).\\
\end{cases}
\eeq
We note that $g_1,g_2,g_4\in \mathcal N^\perp$. 
\subsection{Function spaces}\label{norm-sec}
We will construct the solution $f_R$ using the energy norm
\beq\label{energy-inf}
\bega 
\mathcal E(t)&=\|f_R(t)\|_{L^2_{x,v}}^2+\|\DF f_R(t)\|_{L^2_{x,v}}^2
\enda
\eeq
and the dissipation norm
\beq\label{Dt}
\mathcal D(t)=\frac{1}{\kappa\eps^2}\int_0^t \lw\{\|\sqrt\nu (I-P)f_R(s)\|_{L^2_{x,v}}^2+\|\sqrt\nu (I-P)\DF f_R(s)\|_{L^2_{x,v}}^2
\rw\}ds.
\eeq
Here we denote $\nabla_x\in \{\pt_{x_1},\pt_{x_2}\}$.
We define the weight functions \[
\mathfrak m(v)=e^{\rho_0 |v|^2},\qquad \mathfrak m'(v)=\frac{\mathfrak m(v)}{\nu(v)}\]
for some constant $\rho_0>0$ to be taken later.
We define 
\[
\|\mathfrak m f_R\|_{L^\infty_{t,x,v}}=\|\mathfrak m f_R(s)\|_{L^\infty(0,t,L^\infty_{x,v})},
\]
\[
\|\mathfrak m' f_R\|_{L^2_t L^\infty_{x,v}}=\|\mathfrak m ' f_R(s)\|_{L^2(0,t,L^\infty_{x,v})}.
\]
\subsection{$L^2$ estimates}\label{L2}
Our main Proposition in this section is as follows
\begin{proposition}\label{mainL2} Let $f_R$ be the solution to the equation \eqref{fR-eq}. 
Let $\mathcal E(t)$ and $\mathcal D(t)$ be defined in \eqref{energy-inf} and \eqref{Dt}. There holds 
\[
\mathcal E(t)+\mathcal D(t)\lesssim \frac{\delta^2}{\kappa} \|\mathfrak m'f_R\|^2_{L^2_tL^\infty_{x,v}}\cdot \sup_{0\le s\le t}\mathcal E(s)+\bar c_\kappa \int_0^t \mathcal E(s)ds+\bar c_\kappa\frac{\eps^4}{\delta^2}+\mathcal E(0).
\]
\end{proposition}
\begin{proof}
We recall that 
\[
\mathcal E(t)=\|f_R(t)\|_{L^2_{x,v}}^2+\|\DF f_R(t)\|_{L^2_{x,v}}^2.\\
\]
Taking $\DF$ on both sides of \eqref{fR-eq1} and then multiplying by $\DF f_R$ and integrating in $(x,v)\in \R_x^2\times \R_v^3$, we have 
\[\bega
\frac 1 2 \|\DF f_R(t)\|_{L^2_{x,v}}^2+\frac{1}{\kappa\eps^2}\int_0^t\la L\nabla_x f_R,\nabla_x f_R\ra_{L^2_{x,v}}ds&= \la \DF g,\DF f_R\ra_{L^2_{t,x,v}}.
\enda 
\]
Using the inequality \eqref{Lff}, we have 
\[
\frac 1 {\kappa\eps^2}\int_0^t \la L\nabla_x f_R,\nabla_x f_R\ra\ge O(1)\mathcal D(t)\qquad\text{and}\quad \frac 1 {\kappa\eps^2}\int_0^t \la L f_R, f_R\ra\ge O(1)\mathcal D(t).\]
We recall that $g=\sum_{m=1}^4g_m$ where $g_m$ is given by the equation \eqref{g123}.
First we will show that 
\[
\lw|\la g_1, f_R\ra_{L^2_{t,x,v}}\rw|+\lw|\la \DF g_1,\DF f_R\ra_{L^2_{t,x,v}}\rw|\lesssim \frac{\eps^4}{\delta^2}c_\kappa+o(1)\mathcal D(t).
\]
We shall consider the term $\frac{2\eps}{\kappa\delta}\Gamma(f_2,f_2)$ appearing in $g_1$ only, as this is the most singular term, and the others are treated similarly. Since $\Gamma(f_2,f_2)\in \mathcal N^\perp$, we have 
\[
\frac{2\eps }{\kappa\delta}\la \DF \Gamma(f_2,f_2),\DF f_R\ra_{L^2_{t,x,v}}= \frac{2\eps }{\kappa\delta}\la \DF \Gamma(f_2,f_2),(I-P)\DF f_R\ra_{L^2_{t,x,v}}\lesssim \bar c_\kappa \frac{\eps^4}{\delta^2}+o(1)\mathcal D(t).
\]
Similarly, we have 
\[
\frac {2\eps}{\kappa\delta}\la\Gamma(f_2,f_2),f_R\ra\lesssim \bar c_\kappa \frac{\eps^4}{\delta^2}+o(1)\mathcal D(t).
\]
The proof for $g_1$ is complete. Here we recall that $\bar c_\kappa$ captures the norm of the Navier-Stokes solutions and the next order terms in the Hilbert expansions.
Next, we will show that 
\[
\lw|\la g_2,f_R\ra_{L^2_{t,x,v}}\rw|+\lw|\la \DF g_2,\DF f_R\ra_{L^2_{t,x,v}}\rw|\lesssim \bar c_\kappa \int_0^t \mathcal E(s)ds+ o(1)\mathcal D(t).
\]
We will give the bound for $\frac 2 {\kappa\eps}\la \Gamma(f_1,f_R),f_R\ra_{L^2_{t,x,v}}$ in $g_2$ only, as the other term is similar. Using Lemma \ref{BiGamma} for $L^2_v$, then putting $L^\infty_x$ on $f_1$ and $L^2_x$ on $f_R$, we obtain 
\[\bega 
\frac{1}{\kappa\eps}\la  \Gamma(f_1,f_R),f_R\ra_{L^2_{t,x,v}}&\lesssim \frac{1}{\kappa}\bar c_\kappa^2 \mathcal E(t)+o(1)\mathcal D(t),\\
&\lesssim \bar c_\kappa \int_0^t \mathcal E(s)ds+o(1)\mathcal D(t).
\enda
\]
where we recall remark \ref{c-bar-remark}.
The proof for $g_2$ is complete.
Now, we show that 
\[
\lw|\la g_3, f_R\ra_{L^2_{t,x,v}}\rw|+\lw|\la \DF g_3,\DF f_R\ra_{L^2_{t,x,v}}\rw|\lesssim \frac{\eps^2}{\delta} \bar c_\kappa \lw(\int_0^t \mathcal E(s)ds\rw)^{\frac 1 2 }.
\]
To this end, we simply use Holder inequality to get 
\[\bega 
\la g_3,f_R\ra_{L^2_{t,x,v}}&\lesssim \|g_3\|_{L^2_{t,x,v}}\|f_R\|_{L^2_{t,x,v}}\lesssim \frac{\eps}{\delta} \bar c_\kappa \lw(\int_0^t \mathcal E(s)ds
\rw)^{\frac 1 2 },\\
&\lesssim \frac{\eps^2}{\delta^2}\bar c_\kappa+\int_0^t\mathcal E(s)ds.
\enda \]
The bound for derivative is the same. The proof for $g_3$ is complete.~\\
Lastly, for $g_4$, we show that 
\[
\frac{\delta}{\kappa\eps}\la \DF \Gamma(f_R,f_R), \DF f_R\ra_{L^2_{t,x,v}}\lesssim \frac{\delta^2}{\kappa}\|\mathfrak m f_R\|_{L^2(0,t,L^\infty_{x,v})}^2\cdot 
\sup_{0\le s\le t} \mathcal E(s)+o(1)\mathcal D(t).
\]
To this end, first by Holder inequality, we have
\beq\label{f4}
\bega
\frac{\delta}{\kappa\eps}\la \DF \Gamma(f_R,f_R), \DF f_R\ra_{L^2_{t,x,v}}&\lesssim \frac{\delta^2}{\kappa}\|\DF \Gamma(f_R,f_R)\|_{L^2_tL^2_{x,v}}^2+o(1)\mathcal D(t).\\
\enda
\eeq
Now for the first term on the right hand side above, we have
\[
\|\DF \Gamma(f_R,f_R)\|_{L^2_vL^2_x}\lesssim \|f_R\|_{L^2_v L^\infty_x}\|\DF f_R\|_{L^2_v L^2_x}\lesssim \|\mathfrak m f_R\|_{L^\infty_{x,v}}\|\DF f_R\|_{L^2_{x,v}}.
\]
Hence we get 
 \beq\label{f41}
\|\DF \Gamma(f_R,f_R)\|_{L^2((0,t);L^2_{x,v})}\lesssim \|\mathfrak m f_R\|_{L^2(0,t;L^\infty_{x,v})}\|\DF f_R\|_{L^\infty(0,t,L^2_{x,v})}.
\eeq
Combining \eqref{f4} and \eqref{f41}, we get the result. The proof is complete.
\end{proof}
\subsection{$L^2_tL^\infty_{x,v}$ and $L^\infty_{t,x,v}$ estimates}\label{Linfty}
We recall the norms defined in section \ref{norm-sec}. Our goal of this section is to derive estimates for the following two quantities
\[
\|\mathfrak m f_R\|_{L^\infty_{t,x,v}}^2\qquad \text{and}\qquad \|\mathfrak m' f_R\|_{L^2_tL^\infty_{x,v}}^2.
\]
To highlight our main points, we write the simplified equations, keeping only the most singular terms in \eqref{fR-eq}. Namely, after constructing the approximate solution $f_a$ of the Boltzmann equation, the remainder equation has the form:\[
\pt_t f_R+\frac 1 \eps v\cdot\nabla_x f_R+\frac 1 {\kappa\eps^2}Lf_R=\frac{2}{\kappa\eps}\Gamma(f_1,f_R)+\frac{\delta}{\kappa\eps}\Gamma(f_R,f_R)+\text{other terms}.\]
In order to treat the nonlinear term in the $L^2$ estimate, it is crucial to bound the $L^\infty$ norm of $f_R$, namely by $L^2_t L^\infty_{x,v}$ and $L^\infty_t L^2_{x,v}$. To this end, we use the $L^p-L^\infty$ estimates for the Boltzmann equations (see \cite{Guo-bounded-domain,JangKimAPDE} and references therein). This is due to a hidden mixing effect coming from the transport operator $\frac 1 \eps v\cdot\nabla_x+\frac 1 {\kappa\eps^2}\nu(v)$ and the collision Boltzmann operator $\mathcal K$. The idea can be summarized as follows: Using the fact that $Lf_R=\nu(v)f_R+\mathcal Kf_R$ where $\nu(v)\sim 1+|v|$ and $\mathcal Kf(v)$ is a compact operator, we rewrite the equation in the form 
\[
\pt_t f_R+\frac 1 \eps v\cdot\nabla_x f_R+\frac 1 {\kappa\eps^2} \nu(v)f_R=\mathcal Kf+\text{forcing term}.
\]
We then write the solution as follows
\[\bega 
f_R(t,x,v)&=\frac{1}{\kappa\eps^2}\int_0^t e^{-\frac{\nu(v)(t-s)}{\kappa\eps^2}}\frac{1}{\kappa\eps^2}\mathcal K f_R\lw(s,x-\frac{t-s}{\eps}v,v\rw)ds\\
&\quad+\text{terms involving initial data and forcing terms.}
\enda \]
Rewriting every term involving $f$ in $\mathcal Kf_R(s,x-\frac{t-s}{\eps}v,v)$ once again, with the use of the transport equation, we have the main key term 
\[\bega 
f_R(t,x,v)&=\frac{1}{\kappa^2\eps^4}\int_0^t \int_{\R^3} \int_0^s \int_{\R^3} e^{-\frac{\nu(v)(t-s)}{\kappa\eps^2}}K(v,v_\star) e^{-\frac{\nu(v_\star)(s-\wtd s)}{\kappa\eps^2}}K(v_\star,\wtd v_\star)\\
&\times f_R\lw(\wtd s,x-\frac{t-s}{\eps}v-\frac{s-\wtd s}{\eps}v_\star,\wtd v_\star
\rw)
d\wtd v_\star d\wtd s dv_\star ds+\text{other terms}.
\enda
\]
When any of the velocities $|v_\star|$ or $|\wtd v_\star|$ is large, the kernel on the first line is decaying at least polynomially in these velocities, therefore we  can formally bound this term by $o(1)\|f_R\|_{L^\infty}$ and absorb it to the left hand side in the $L^\infty$ estimate. This is the same when $|s-\wtd s|=o(1)\kappa\eps^2$. In the set where $|v_\star|$ and $|\wtd v_\star|$ are bounded, we use Holder inequality to get 
\[
|f_R|\lesssim \lw\{\int_{|v_\star|, |\wtd v_\star|\lesssim 1}|Pf_R|^p\lw(\wtd s,x-\frac{t-s}{\eps}v-\frac{s-\wtd s}{\eps}v_\star,\wtd v_\star\rw)dv_\star d\wtd v_\star\rw\}^{1/p}.
\]
Then, using the change of variables $(v^1_\star,v_\star^2,\wtd v_\star)\to (\wtd x_1,\wtd x_2,\wtd v_\star)$ where 
\[
dv_\star d\wtd v_\star =\lw(\frac{\eps}{s-\wtd s}
\rw)^2d \wtd x d\wtd v_\star\lesssim \frac{\eps^2}{(\delta_1\kappa\eps^2)^2}d\wtd xd\wtd v_\star,\qquad \wtd x=x-\frac{t-s}{\eps}v-\frac{s-\wtd s}{\eps}v_\star.
\]
we get the upper bound $ \frac{1}{(\kappa\eps)^{\frac 2 p}}\|Pf_R(\wtd s)\|_{L^p_x}$, which is roughly the $L^p$ norm of $f_R$. Next, the precise estimates follow.
Our main proposition is as follows
\begin{proposition} \label{mainsec3}Let $p>1$. If $c_\kappa\eps\ll 1$, 
there holds 
\[\bega 
\|\mathfrak mf_R\|_{L^\infty_{t,x,v}}&\lesssim  \|\mathfrak m f_{0,R}\|_{ L^\infty_{x,v}}+\frac{\bar c_\kappa\eps^3}{\delta}+\bar c_\kappa\eps\|\mathfrak m f_R\|_{L^\infty_{t,x,v}}+\delta\eps\|\mathfrak mf_R\|_{L^\infty_{t,x,v}}^2 \\
&\quad +\frac{1}{\kappa^{\frac 1 2+\frac 2 p}}\frac{1}{\eps^{1+\frac 2 p}}\|Pf_R\|_{L^2_t L^{p}_x}
+\kappa^{-1/2}\sqrt{\mathcal D(t)}.
\enda \]
\[\bega 
\|\mathfrak m'f_R\|_{L^2(0,t,L^\infty_{x,v})}&\lesssim  \|\mathfrak m f_{0,R}\|_{ L^\infty_{x,v}}+\frac{\bar c_\kappa\eps^3}{\delta}+\bar c_\kappa\eps\|\mathfrak m f_R\|_{L^\infty_{t,x,v}}+\delta\eps\|\mathfrak mf_R\|_{L^\infty_{t,x,v}}^2 \\
&\quad+\frac{1}{(\eps \kappa)^{\frac 2 p}}\|Pf_R\|_{L^2(0,t,L^p_x)}
+\kappa^{-1/2}\sqrt{\mathcal D(t)}.
\enda \]
\end{proposition}
\begin{proof}
From the equation \eqref{fR-eq}, we have 
\[
\pt_t f_R+\frac 1 \eps v\cdot\nabla_x f_R+\frac 1 {\kappa\eps^2}\nu(v)f=\frac{1}{\kappa\eps^2}\mathcal K f+g(t,x,v)
\]
where $g$ denotes the right hand side of the equation \eqref{fR-eq}. By applying the Duhamel principle, we have 
\[\bega
f_R(t,x,v)&=e^{-\frac{\nu(v) t}{\kappa\eps^2}}f_{0,R}\lw(x-\frac{tv}{\eps},v\rw)\\
&\quad+\int_0^t e^{-\frac{\nu(v)(t-s)}{\kappa\eps^2}}\frac{1}{\kappa\eps^2}\mathcal K f_R\lw(s,x-\frac{t-s}{\eps}v,v\rw)ds\\
&\quad+\int_0^t e^{-\frac{\nu(v)(t-s)}{\kappa\eps^2}}g\lw(s,x-\frac{t-s}{\eps}v,v
\rw)ds.
\enda 
\]
Using the definition of the linear operator $\mathcal K f_R$, we have, for $s\in [0,t]$, 
\[
\mathcal Kf_R\lw(s,x-\frac{t-s}{\eps}v,v
\rw)=\int_{\R^3}K(v,v_\star)f_R\lw(s,x-\frac{t-s}{\eps}v,v_\star
\rw)d v_\star.
\]
Applying the Duhamel formula again for $f_R(s,x-\eps^{-1}(t-s)v,v_\star)$, we obtain 
\[
f_R(t,x,v)=\mathcal M_0 (f_{0,R})+\mathcal M_1(g)+\mathcal M_2(f_R)
\]
where the linear operators $\mathcal M_0, \mathcal M_1$ and $\mathcal M_3$ is defined by 
\beq\label{M0}
\bega 
\mathcal M_0(f_{0,R})&=e^{-\frac{\nu(v) t}{\kappa\eps^2}}f_{0,R}\lw(x-\frac{tv}{\eps},v\rw)\\
&\quad+\frac{1}{\kappa\eps^2}\int_0^t\int_{\R^3} e^{-\frac{\nu(v)(t-s)}{\kappa\eps^2}}K(v,v_\star) e^{-\frac{\nu(v_\star)s}{\kappa\eps^2}}f_{0,R}\lw(x-\frac{t-s}{\eps}v-\frac{s}{\eps}v_\star,v_\star
\rw)dv_\star ds,
\enda 
\eeq
\beq\label{M13}
\bega 
\mathcal M_1(g)&=\int_0^t e^{-\frac{\nu(v)(t-s)}{\kappa\eps^2}}g\lw(s,x-\frac{t-s}{\eps}v,v
\rw)ds\\
&\quad+\frac{1}{\kappa\eps^2}\int_0^t \int_{\R^3}\int_0^s e^{-\frac{\nu(v)(t-s)}{\kappa\eps^2}}K(v,v_\star) e^{-\frac{\nu(v_\star)(s-\wtd s)}{\kappa\eps^2}}g\lw(\wtd s,x-\frac{t-s}{\eps}v-\frac{s-\wtd s}{\eps}v_\star,v_\star
\rw)d\wtd s  dv_\star ds,\\
\mathcal M_2(f_R)&=\frac{1}{\kappa^2\eps^4}\int_0^t \int_{\R^3} \int_0^s \int_{\R^3} e^{-\frac{\nu(v)(t-s)}{\kappa\eps^2}}K(v,v_\star) e^{-\frac{\nu(v_\star)(s-\wtd s)}{\kappa\eps^2}}K(v_\star,\wtd v_\star)\\
&\times f_R\lw(\wtd s,x-\frac{t-s}{\eps}v-\frac{s-\wtd s}{\eps}v_\star,\wtd v_\star
\rw)
d\wtd v_\star d\wtd s dv_\star ds.\enda
\eeq
\end{proof}
In the next propositions, we give the bounds for the linear operators $\mathcal M_0, \mathcal M_1$ and $\mathcal M_2$.
\begin{proposition}
There holds 
\[\bega 
\|\mathfrak m \mathcal M_0(f_{0,R})\|_{L^\infty_{t,x,v}}&\lesssim  \|\mathfrak m f_{0,R}\|_{L^\infty_{x,v}},\\
\|\mathfrak m'  \mathcal M_0(f_{0,R})\|_{L^2_t L^\infty_{x,v}}&\lesssim \sqrt\kappa\eps  \|\mathfrak m f_{0,R}\|_{L^\infty_{x,v}}.
\enda 
\]
\end{proposition}
\begin{proof}
For the first term on the right hand side of \eqref{M0}, it is straight forward that 
\[\bega 
\lw\|e^{-\frac{\nu(v) t}{\kappa\eps^2}}\mathfrak m (v)f_{0,R}\lw(x-\frac{tv}{\eps},v\rw)\rw\|_{L^2_tL^\infty_{x,v}}&\lesssim \sqrt\kappa\eps \|\mathfrak m f_{0,R}\|_{L^\infty_{x,v}},\\
\lw\|e^{-\frac{\nu(v) t}{\kappa\eps^2}}\mathfrak m (v)f_{0,R}\lw(x-\frac{tv}{\eps},v\rw)\rw\|_{L^\infty_{t,x,v}}&\lesssim  \|\mathfrak m f_{0,R}\|_{L^\infty_{x,v}}.\\
\enda 
\]
On the other hand, we bound the second term in \eqref{M0} as follows
\[\bega
&\mathfrak m(v)\frac{1}{\kappa\eps^2}\lw|\int_0^t\int_{\R^3} e^{-\frac{\nu(v)(t-s)}{\kappa\eps^2}}K(v,v_\star) e^{-\frac{\nu(v_\star)s}{\kappa\eps^2}}f_{0,R}\lw(x-\frac{t-s}{\eps}v-\frac{s}{\eps}v_\star,v_\star
\rw)dv_\star ds\rw|\\
&\lesssim\|\mathfrak m f_{0,R}\|_{L^\infty_{x,v}}
 \int_0^t e^{-\frac{\nu(v)(t-s)}{\kappa\eps^2}}ds\int_{\R^3}|K(v,v_\star)|\frac{\mathfrak m(v)}{\mathfrak m(v_\star)}dv_\star, \\
 &\lesssim\|\mathfrak m f_{0,R}\|_{L^\infty_{x,v}} \nu(v) \int_0^t e^{-\frac{\nu(v)(t-s)}{\kappa\eps^2}}ds.\enda 
\]
Here we used the inequality \eqref{K-ineq1}. The proof is complete.
\end{proof}
\begin{lemma} \label{M1-lem}Let $\mathcal M_1$ is the linear operator defined in \eqref{M13}. There holds 
\[\begin{cases}
&\|\mathfrak m \mathcal M_1(g)\|_{L^\infty_{t,x,v}} \lesssim \kappa\eps^2\|\mathfrak m' g\|_{L^\infty_{t,x,v}},\\
&\|\mathfrak m' \mathcal M_1(g)\|_{L^2_t L^\infty_{x,v}} \lesssim \kappa\eps^2\|\mathfrak m' g\|_{L^2_tL^\infty_{x,v}}.
\end{cases}
\]
\end{lemma}
\begin{proof}
We have \[\bega
&\mathfrak m(v)\mathcal M_1(g)\\
&=\nu(v)\int_0^t e^{-\frac{\nu(v)(t-s)}{\kappa\eps^2}}\mathfrak m'(v) g\lw(s,x-\frac{t-s}{\eps}v,v
\rw)ds\\
&+\nu(v)\frac{1}{\kappa\eps^2}\int_0^t \int_{\R^3}\int_0^s e^{-\frac{\nu(v)(t-s)}{\kappa\eps^2}}K(v,v_\star) e^{-\frac{\nu(v_\star)(s-\wtd s)}{\kappa\eps^2}}\mathfrak m'(v) g\lw(\wtd s,x-\frac{t-s}{\eps}v-\frac{s-\wtd s}{\eps}v_\star,v_\star
\rw)d\wtd s  dv_\star ds.\enda 
\]
This implies
\[\bega 
\mathfrak m(v)\mathcal M_1(g)&\lesssim\nu(v)\int_0^t e^{-\frac{\nu(v)(t-s)}{\kappa\eps^2}} \|\mathfrak m' g(s)\|_{L^\infty_{x,v}} ds\\
&\quad+\nu(v)\frac{1}{\kappa\eps^2}\int_0^t \int_0^s e^{-\frac{\nu(v)(t-s)}{\kappa\eps^2}}e^{-\frac{\nu_0(s-\wtd s)}{\kappa\eps^2}}\|\mathfrak m' g(\wtd s)\|_{L^\infty_{x,v}}d\wtd s ds\cdot \int_{\R^3}|K(v,v_\star)| \frac{\mathfrak m'(v)}{\mathfrak m'(v_\star)}dv_\star.\enda 
\]
Again using Lemma \ref{K-ineq1} and the convolution in time estimate, we get 
\[\begin{cases}
&\|\mathfrak m \mathcal M_1(g)\|_{L^\infty_{t,x,v}} \lesssim \kappa\eps^2\|\mathfrak m' g\|_{L^\infty_{t,x,v}},\\
&\|\mathfrak m' \mathcal M_1(g)\|_{L^2_t L^\infty_{x,v}} \lesssim \kappa\eps^2\|\mathfrak m' g\|_{L^2_tL^\infty_{x,v}}.
\end{cases}
\]
The proof is complete.
\end{proof}
\begin{proposition}\label{mg}
Let $g$ be the forcing term defined in \eqref{g123}. There holds 
\[
\|\mathfrak m' g\|_{L^\infty_{t,x,v}}\lesssim \frac{\bar c_\kappa\eps}{\delta}+\frac{\bar c_\kappa}{\eps} \|\mathfrak mf_R\|_{L^\infty_{t,x,v}}+\frac{\delta}{\kappa\eps}\|\mathfrak m f_R\|_{L^\infty_{t,x,v}}^2.\]
\end{proposition}
\begin{proof}
We will treat the terms 
\[
\frac{2\eps}{\kappa\delta}\Gamma(f_2,f_2),\quad \frac{2}{\kappa\eps}\Gamma(f_1,f_R),\quad -\frac{\eps^2}{\delta}\pt_t f_3 \quad\text{and}\quad  \frac{\delta}{\kappa\eps}\Gamma(f_R,f_R)\]
only, since these are the most singular terms and the others can be treated similarly. Using lemma \ref{Gamma-inf}, we have 
\[\bega 
&\frac{2\eps}{\kappa\delta}\|\Gamma(f_1,f_2)\|_{L^\infty_{x,v}}+\frac{2}{\kappa\eps}\|\Gamma(f_1,f_R)\|_{L^\infty_{x,v}}+ \frac{\delta}{\kappa\eps}\|\Gamma(f_R,f_R)\|_{L^\infty_{x,v}}\\&\lesssim \frac{\eps}{\kappa\delta}\bar c_\kappa^2+\frac{1}{\kappa\eps}c_\kappa \|\mathfrak mf_R\|_{L^\infty_{x,v}}+\frac{\delta}{\kappa\eps}\|\mathfrak m f_R\|_{L^\infty_{x,v}}^2.
\enda 
\]
On the other hand, it is clear that 
\[
\frac{\eps^2}{\delta}\|\pt_t f_3\|_{L^\infty_{x,v}}\lesssim \frac{\eps^2}{\delta}\bar c_\kappa.
\]
The proof is complete.
\end{proof}
Combing Lemma \ref{M1-lem} and Proposition \ref{mg}, we have the following corollary:
\begin{corollary} Let $g$ be defined in \eqref{g123}. There holds 
\[
\|\mathfrak m \mathcal M_1(g)\|_{L^\infty_{t,x,v}}+\|\mathfrak m'  \mathcal M_1(g)\|_{L^2_t L^\infty_{x,v}}\lesssim  \bar c_\kappa\frac{\eps^3}{\delta}+\bar c_\kappa\eps\|\mathfrak m f_R\|_{L^\infty_{t,x,v}}+\delta\eps\|\mathfrak mf_R\|_{L^\infty_{t,x,v}}^2. \\
\]
\end{corollary}
Next, we bound the operator $\mathcal M_2(f_R)$ defined in \eqref{M13}. We have 
\begin{proposition} There holds 
\[
\|\mathfrak  m \mathcal M_2 (f_R)\|_{L^\infty_{t,x,v}}\lesssim o(1)\cdot \|\mathfrak m f_R\|_{L^\infty_{t,x,v}}+\frac{1}{\kappa^{\frac 2 p+\frac 1 2}}\frac{1}{\eps^{1+\frac 2 p}}
\|Pf_R\|_{L^2_t L^p_x}
+\kappa^{-1/2}\sqrt{\mathcal D(t)}
\]
and 
\[
\|\mathfrak  m' \mathcal M_2 (f_R)\|_{L^2_tL^\infty_{x,v}}\lesssim o(1)\cdot \|\mathfrak m' f_R\|_{L^2_t L^\infty_{x,v}}+\frac{1}{(\kappa\eps)^{\frac 2 p}}
\|Pf_R\|_{L^2_t L^p_x}
+\kappa^{-1/2}\sqrt{\mathcal D(t)},
\]
where $\mathcal D(t)$ is given in \eqref{Dt}.\end{proposition}
\begin{proof}
First, we see that 
\beq\label{f5}
\bega 
\mathfrak m(v) \mathcal M_2 (f_R)=&\frac{1}{\kappa^2\eps^4}\int_0^t \int_{\R^3} \int_0^s \int_{\R^3} e^{-\frac{\nu(v)(t-s)}{\kappa\eps^2}} e^{-\frac{\nu(v_\star)(s-\wtd s)}{\kappa\eps^2}}K(v,v_\star)K(v_\star,\wtd v_\star)\frac{\mathfrak m(v)}{\mathfrak m(\wtd v_\star)} \\
&\times \mathfrak m (\wtd v_\star) f_R\lw(\wtd s,x-\frac{t-s}{\eps}v-\frac{s-\wtd s}{\eps}v_\star,\wtd v_\star
\rw)
d\wtd v_\star d\wtd s dv_\star ds.
\enda 
\eeq
Take $\delta_1>0$ to be small and $\bar N>0$ to be a large number. We decompose the integral \[
\int_0^t \int_0^s \int_{\R^3_{\wtd v_\star}\times\R^3_{v_\star}} d\wtd v_\star dv_\star d\wtd s ds\] as follows
\beq\label{split}
\bega 
\int_0^t \int_0^s \int_{\R^3_{\wtd v_\star}\times\R^3_{v_\star}} d\wtd v_\star dv_\star d\wtd s ds&=\int_0^t \int_{|s-\wtd s|\le \delta_1\kappa\eps^2}\int_{\R^3_{\wtd v_\star}\times\R^3_{v_\star}} d\wtd v_\star dv_\star
 d\wtd s ds\\
 &\quad +\int_0^t \int_{|s-\wtd s|\ge \delta_1\kappa\eps^2}\int_{\max\{|v|,|v-v_\star|,|v-\wtd v_\star|
 \}\ge \bar N} d\wtd v_\star dv_\star
 d\wtd s ds\\
 &\quad+\int_0^t \int_{|s-\wtd s|\ge \delta_1\kappa\eps^2}\int_{\max\{|v|,|v-v_\star|,|v-\wtd v_\star|
 \}\le \bar N} d\wtd v_\star dv_\star
 d\wtd s ds.\\
  \enda 
\eeq
We have 
\[
\mathfrak m(v) \mathcal M_2(f_R)(t,x,v)=\mathfrak m f_{R,1}+\mathfrak m f_{R,2}+\mathfrak m f_{R,3},
\]
where $f_{R,i}$ is respectively the integration in \eqref{f5} integrated over the splitting of the regions defined in \eqref{split}. 
First, using the fact that 
\[\frac{1}{\kappa\eps^2}\int_0^t e^{-\frac{\nu(v)(t-s)}{\kappa\eps^2}}ds\lesssim 1,\quad \frac 1 {\kappa\eps^2}\int_{|s-\wtd s|\le \delta_1\kappa\eps^2}1ds\lesssim \delta_1,\quad \int_{\R^3\times \R^3}\frac{\mathfrak m(v)}{\mathfrak m (\wtd v_\star)}\lw|K(v,v_\star)K(v_\star,\wtd v_\star)\rw|d\wtd v_\star dv_\star\lesssim 1.
\]
we get 
\[
\mathfrak  m(v)f_{R,1}\lesssim  \delta_1 \|\mathfrak m f_R\|_{L^\infty_{t,x,v}}.\]
Next we bound $\mathfrak m f_{R,2}$. Indeed, we will show that 
\[
\mathfrak m f_{R,2}\lesssim \frac{1}{\bar N+1}\|\mathfrak m f_R\|_{L^\infty_{t,x,v}}.
\]
To this end, it suffices to show that 
\[
\int_{\max\{|v|,|v-v_\star|,|v_\star-\wtd v_\star|\}\ge \bar N}\lw|K(v,v_\star)K(v_\star,\wtd v_\star)\rw|\frac{\mathfrak m(v)}{\mathfrak m(\wtd v_\star)}d\wtd v_\star dv_\star\lesssim \frac{1}{1+\bar N}.
\]
If 
 $|v|\ge \bar N$,  from \eqref{K-ineq1} we get 
\[
\int_{\R^3\times\R^3}\lw|K(v,v_\star)K(v_\star,\wtd v_\star)\rw|\frac{\mathfrak m(v)}{\mathfrak m(\wtd v_\star)}d\wtd v_\star dv_\star\lesssim \int_{\R^3}\lw|K(v,v_\star)\rw|\frac{\mathfrak m(v)}{\mathfrak m(v_\star)}dv_\star\lesssim \frac 1 {1+|v|} \lesssim \frac 1 {1+\bar N}.
\]
Now if  $\max\{|v-v_\star|,|v_\star-\wtd v_\star|\}\ge \bar N$, we use the Grad estimate \eqref{grad} to get, for all $\vartheta\in (0,\frac 1 8)$:
\[\bega 
&\int_{\max\{|v-v_\star|,|v_\star-\wtd v_\star|\}\ge \bar N}\lw|K(v,v_\star)K(v_\star,\wtd v_\star)\rw|\frac{\mathfrak m(v)}{\mathfrak m(\wtd v_\star)}d\wtd v_\star dv_\star\\
&\lesssim \int_{\max\{|v-v_\star|,|v_\star-\wtd v_\star|\}\ge \bar N}  \lw|K_{\vartheta}(v,v_\star)K_{\vartheta}(v_\star,\wtd v_\star)\rw|\frac{\mathfrak m(v)}{\mathfrak m(\wtd v_\star)}d\wtd v_\star dv_\star\\
&\lesssim e^{-\frac{\vartheta}{2}\bar N^2}\int_{\R^3\times\R^3}\lw|K_{\vartheta/2}(v,v_\star)K_{\vartheta/2}(v_\star,\wtd v_\star)\rw|\frac{\mathfrak m(v)}{\mathfrak m(\wtd v_\star)}d\wtd v_\star dv_\star\\
&\lesssim e^{-\frac{\vartheta}{2}\bar N^2}\lesssim \frac{1}{1+\bar N}.
\enda 
\]
Finally, we bound $\mathfrak m f_{R,3}$. Since $
\max\{|v|,|v-v_\star|,|v-\wtd v_\star|
 \}\le \bar N$, we have 
 $|v_\star|\le 2\bar N$ and $|\wtd v_\star|\le 3\bar N$. Using the decomposition $f_R=Pf_R+(I-P)f_R$, we  bound the integral $dv_\star d\wtd v_\star$ in $\mathfrak m f_{R,3}$ as follows
\[\bega 
&\int_{|v_\star|\le 2\bar N,|\wtd v_\star|\le 3\bar N} 
\frac{\mathfrak m(v)}{\mathfrak m (\wtd v_\star)}K(v,v_\star)K(v_\star,\wtd v_\star)\mathfrak m(\wtd v_\star) f_R\lw(\wtd s,x-\frac{t-s}{\eps}v-\frac{s-\wtd s}{\eps}v_\star,\wtd v_\star
\rw)
dv_\star d\wtd v_\star\\
&\lesssim  \lw\{\int_{|v_\star|\le 2\bar N,|\wtd v_\star|\le 3\bar N}|Pf_R|^p\lw(\wtd s,x-\frac{t-s}{\eps}v-\frac{s-\wtd s}{\eps}v_\star,\wtd v_\star\rw)dv_\star d\wtd v_\star\rw\}^{1/p}\\
&\quad+\lw\{\int_{|v_\star|\le 2\bar N,|\wtd v_\star|\le 3\bar N}|(I-P)f_R|^{2}\lw(\wtd s,x-\frac{t-s}{\eps}v-\frac{s-\wtd s}{\eps}v_\star,\wtd v_\star\rw)dv_\star d\wtd v_\star\rw\}^{\frac 1 2}\\
&\lesssim \frac{1}{(\kappa\eps)^{\frac 2 p}}\|Pf_R(\wtd s)\|_{L^p_x}+\frac{1}{\kappa \eps}\|(I-P)f_R(\wtd s)\|_{L^{2}_{x,v}}.
\enda 
\]
Here, we have use the fact that \[
dv_\star d\wtd v_\star =\lw(\frac{\eps}{s-\wtd s}
\rw)^2d \wtd x d\wtd v_\star\lesssim \frac{\eps^2}{(\delta_1\kappa\eps^2)^2}d\wtd xd\wtd v_\star,\qquad \wtd x=x-\frac{t-s}{\eps}v-\frac{s-\wtd s}{\eps}v_\star.
\]
This implies 
\beq \label{mfr3}
\bega 
\mathfrak mf_{R,3}&\lesssim \frac{1}{\kappa^2\eps^4}\int_0^t e^{-\frac{\nu(v)(t-s)}{\kappa\eps^2}}\lw\{\int_0^s e^{-\frac{\nu(v_\star)(s-\wtd s)}{\kappa\eps^2}}\frac{1}{(\kappa\eps)^{\frac 2 p}}\|Pf_R(\wtd s)\|_{L^p_{x}}d\wtd s\rw\} ds\\
&\quad+\frac{1}{\kappa^2\eps^4}\int_0^t e^{-\frac{\nu(v)(t-s)}{\kappa\eps^2}}\lw\{e^{-\frac{\nu(v_\star)(s-\wtd s)}{\kappa\eps^2}}\frac{1}{\kappa \eps}\|(I-P)f_R(\wtd s)\|_{L^{2}_{x,v}}d\wtd s\rw\}ds\\
&\lesssim \frac{1}{\kappa^{\frac 2 p+\frac 1 2}}\frac{1}{\eps^{1+\frac 2 p}}
\|Pf_R\|_{L^2_t L^p_x}
+\kappa^{-1/2}\sqrt{\mathcal D(t)}.\enda 
\eeq
where we used the Holder inequality in time.
As for $\|\mathfrak m' \mathcal M_2(f_R)\|_{L^2_t L^\infty_{x,v}}$, the proof is exactly the same for the first two integrations in \eqref{split}, up until the estimate \eqref{mfr3}. Instead of estimating the integral in time $t$ pointwise, we use the convolution estimate in time, to get 
\[\frac{1}{\kappa^2\eps^4}\lw\|
\int_0^t e^{-\frac{\nu(v)(t-s)}{\kappa\eps^2}}\lw\{\int_0^s e^{-\frac{\nu(v_\star)(s-\wtd s)}{\kappa\eps^2}}\frac{1}{(\kappa\eps)^{\frac 2 p}}\|Pf_R(\wtd s)\|_{L^p_{x}}d\wtd s\rw\} ds\rw\|_{L^2_t}\lesssim \frac{1}{(\kappa\eps)^{\frac 2 p}}\|Pf_R\|_{L^2_tL^p_x}
\]
and 
\[\frac{1}{\kappa^2\eps^4}\lw\|
\int_0^t e^{-\frac{\nu(v)(t-s)}{\kappa\eps^2}}\lw\{e^{-\frac{\nu(v_\star)(s-\wtd s)}{\kappa\eps^2}}\frac{1}{\kappa \eps}\|(I-P)f_R(\wtd s)\|_{L^{2}_{x,v}}d\wtd s\rw\}ds\rw\|_{L^2_t}\lesssim \frac{1}{\kappa\eps}\|(I-P)f_R\|_{L^2_{t,x,v}}.
\]
The proof is complete.

\end{proof}

\subsection{Nonlinear iterations for the remainder }\label{iter}
In this section, we close the estimate for the energy 
$\mathcal E(t)$. Let $p>4$ and choose $\beta>0$ such that 
\[
\frac 1 2 +\frac 2 p<\beta<2
\]
We define 
\beq \label{delta}
\delta=\eps^\beta .
\eeq
\begin{proposition}
If 
\beq\label{L2Linf}
\delta \kappa^{-\frac 1 2}\|\mathfrak m' f_R\|_{L^2_t L^\infty_{x,v}}\ll 1
\eeq
then 
\beq\label{ineqE1}
\sup_{0\le s\le t}\mathcal E(s)+\mathcal D(t)\le \lw\{C_0+\mathcal E(0)\rw\} \bar c_\kappa e^{C_0\bar c_\kappa t}
\eeq
for some $C_0>0$.
\end{proposition}
\begin{proof} From Proposition \ref{mainL2}, we have 
\[
\mathcal E(t)+\mathcal D(t)\lesssim \frac{\delta^2}{\kappa} \|\mathfrak m'f_R\|^2_{L^2_tL^\infty_{x,v}}\cdot \sup_{0\le s\le t}\mathcal E(s)+\bar c_\kappa \int_0^t \mathcal E(s)ds+\bar c_\kappa\frac{\eps^4}{\delta^2}+\mathcal E(0).
\]
This implies 
\[
\lw(1-C_0\frac{\delta^2}{\kappa}\|\mathfrak m' f_R\|_{L^2_tL^\infty_{x,v}}^2
\rw)\cdot \sup_{0\le s\le t}\mathcal E(s)\lesssim \bar c_\kappa \int_0^t \mathcal E(s)ds+\bar c_\kappa\frac{\eps^4}{\delta^2}+\mathcal E(0).
\]
Hence as long as the inequality \eqref{L2Linf} holds, we have 
\[
\sup_{0\le s\le t}\mathcal E(s)\lesssim \lw\{\bar c_k\frac{\eps^4}{\delta^2}+\mathcal E(0)
\rw\}\lw(1+\bar c_\kappa e^{C_0\bar c_\kappa t}
\rw)
\] 
and for some constant $C_0>0$.  Since $\bar c_\kappa \frac{\eps^4}{\delta^2}=c_\kappa \eps^{4-2\beta} \lesssim 1$ as $\eps\to 0$, we have \eqref{ineqE1}. The proof is complete.\end{proof}
\begin{proposition}
Let 
\[
{\mathcal F}(t)=\frac{\delta}{\sqrt\kappa} \|\mathfrak m' f_R\|_{L^2_tL^\infty_{x,v}}+\delta \sqrt\eps \|\mathfrak mf_R\|_{L^\infty_{t,x,v}},
\]
then ${ \mathcal{F}}(t)$ satisfies the apriori estimate 
\[
{ \mathcal{F}}(t)\lesssim \eps^{\beta}\kappa^{-\frac 1 2}\|\mathfrak mf_{0,R}\|_{L^\infty_{x,v}}+\bar c_\kappa \kappa^{-\frac 1 2}\eps^3+\bar c_\kappa\sqrt\eps \kappa^{-\frac 1 2} { \mathcal{F}}(t)+\kappa^{-\frac 1 2}{ \mathcal{F}}(t)^2+\frac{\eps^{\beta-\frac 2 p}}{\kappa^{\frac 2 p+\frac 1 2}}\sqrt{\mathcal E(t)}+\eps^\beta \kappa^{-\frac 1 2}\sqrt{\mathcal D(t)}.
\]
\end{proposition}
\begin{proof} From Proposition \ref{mainsec3}, we have  
\[\bega 
\|\mathfrak m'f_R\|_{L^2_tL^\infty_{x,v}}&\lesssim  \|\mathfrak m f_{0,R}\|_{ L^\infty_{x,v}}+\frac{\bar c_\kappa\eps^3}{\delta}+\bar c_\kappa\eps\|\mathfrak m f_R\|_{L^\infty_{t,x,v}}\\
&+\delta\eps\|\mathfrak mf_R\|_{L^\infty_{t,x,v}}^2 +\frac{1}{(\eps \kappa)^{\frac 2 p}}\|Pf_R\|_{L^2(0,t,L^p_x)}
+\kappa^{-1/2}\sqrt{\mathcal D(t)}.
\enda 
\]
Hence 
\beq\label{ineq12}
\bega 
\frac{\delta}{\sqrt\kappa}\|\mathfrak m' f_R\|_{L^2_t L^\infty_{x,v}}&\lesssim \frac{\delta}{\sqrt\kappa}\|\mathfrak m f_{0,R}\|_{L^\infty_{x,v}}+\frac{\bar c_\kappa \eps^3}{\sqrt\kappa}+\bar c_\kappa \eps \delta\kappa^{-1/2}\|\mathfrak m f_R\|_{L^\infty_{t,x,v}}+\delta^2\eps \kappa^{-1/2}\|\mathfrak m f_R\|_{L^\infty_{t,x,v}}^2\\
&\quad+\frac{\delta}{\eps^{\frac 2 p}\kappa^{\frac 2 p+\frac 1 2}} \|Pf_R\|_{L^2_tL^p_x}+\frac{\delta}{\kappa}\sqrt {D(t)}.
\enda 
\eeq
By the Sobolev embedding for $p>4$, we have 
\[
\|Pf_R\|_{L^p_x}\lesssim \sqrt{\mathcal E(t)}.
\]
Combining with \eqref{ineq12}, we have 
\beq  \label{ineqmf1}
\bega
\frac{\delta}{\sqrt\kappa}\|\mathfrak m' f_R\|_{L^2_t L^\infty_{x,v}}&\lesssim \delta \kappa^{-1/2}\|\mathfrak m f_{0,R}\|_{L^\infty_{x,v}}+\kappa^{-1/2}\bar  c_\kappa \eps^3 +\bar c_\kappa \frac{\sqrt\eps}{\kappa^{1/2}}d(t)\\
&\quad+\kappa^{-1/2} d(t)^2+\frac{\eps^{\beta-\frac 2 p}}{\kappa^{\frac 2 p+\frac 1 2}}\sqrt{\mathcal E(t)}+\frac{\delta}{\kappa}\sqrt{\mathcal D(t)}.
\enda 
\eeq
On the other hand, we have 
\[\bega 
\|\mathfrak mf_R\|_{L^\infty_{t,x,v}}&\lesssim   \|\mathfrak m f_{0,R}\|_{ L^\infty_{x,v}}+\frac{\bar c_\kappa\eps^3}{\delta}+\bar c_\kappa\eps\|\mathfrak m f_R\|_{L^\infty_{t,x,v}}+\delta\eps\|\mathfrak mf_R\|_{L^\infty_{t,x,v}}^2 \\
&\quad+\frac{1}{\kappa^{\frac 1 2+\frac 2 p}}\frac{1}{\eps^{1+\frac 2 p}}\|Pf_R\|_{L^2_t L^{p}_x}
+\kappa^{-1/2}\sqrt{\mathcal D(t)}.
\enda \]
This implies 
\beq \label{ineqmf2}
\bega
\delta\sqrt\eps \|\mathfrak m f_R\|_{L^\infty_{t,x,v}}&\lesssim \delta\sqrt\eps \|\mathfrak m f_{0,R}\|_{L^\infty_{x,v}}+\bar c_\kappa \eps^{\frac 7 2}+\bar c_\kappa\eps d(t)+\sqrt\eps d(t)^2\\
&\quad+\frac{1}{\kappa^{\frac 1 2+\frac 2 p}}\eps^{\beta-\frac 1 2 -\frac 2 p}\sqrt{\mathcal E(t)}+\frac{\delta\sqrt\eps}{\sqrt\kappa}\sqrt{\mathcal D(t)}.
\enda 
\eeq
\end{proof}
\subsection{Proof of the main theorems}\label{proof-main}
We now give the proof for the pointwise estimates in Theorem \ref{mainthm}. By the apriori estimates given in section \ref{iter}, there exists a local solution $f_R$ to the equation \eqref{fR-eq} satisfying 
\[
\delta \sqrt\eps \|\mathfrak m f_R\|_{L^\infty_{t,x,v}}\lesssim \kappa^{-\frac 1 2-\frac 2 p} \eps^{\beta-\frac 2 p-\frac 1 2}.
\]
Since $c_\kappa\eps\ll1$, we have, for any $\beta'<\beta$, we have 
\[
\delta\sqrt\eps \|\mathfrak m f_R\|_{L^\infty_{t,x,v}}\lesssim \eps^{\beta'-\frac 2 p -\frac 1 2}.
\]
We define $\gamma_0=\beta'-\frac 2 p-\frac 1 2$, then for $\beta\to 2^-$, we can choose $\beta'$ as close to $2$. Therefore, for $p>4$ and $p\to \infty$, the constant
\[
\gamma_1=\beta'-\frac 2 p-\frac 1 2
\]
can be taken arbitrary close to $2-\frac 1 2=\frac 3 2$.
From the expansion of $F$, given by 
\beq\label{expan-final}
F=\mu+\eps\sqrt\mu f_1+\eps^2\sqrt\mu f_2+\eps^3 \sqrt\mu f_3 + \eps\delta \sqrt\mu f_R,
\eeq
we obtain 
\[
|F-\mu-\eps\sqrt\mu (u^\kappa\cdot v)\sqrt\mu|\le \eps^2\sqrt\mu |f_2|+\eps^3 \sqrt\mu |f_3|+ \sqrt\mu \eps^{\frac 1 2+\gamma_1}\mathfrak m(v)^{-1}\lesssim \eps^{\gamma_0}\sqrt\mu e^{-\rho_0|v|^2}
\]
where $\gamma_0=\gamma_1+\frac 1 2<2$. Theorem \ref{mainthm} follows, by combining the above inequality with Theorem \ref{compareNS}. Theorem \ref{macro-vor} follows directly from the expansion \eqref{expan-final} and $\delta \|\nabla_x f_R(t)\|_{L^2_{x,v}}\lesssim \delta\mathcal E(t)\to 0$ as $\eps\to 0^+$. The proof is complete.
\section{Appendix}
In this section, we collect technical lemmas used in this paper. To simplify the notations, we will use Einstein summation  notation in the calculations involving Boltzmann equations.
\begin{lemma}\label{v-dot-Pf}
If $f\in \mathcal N$ and $Pf=\lw(a+b\cdot v+c\frac{|v|^2-3}{2}\rw)\sqrt\mu$. Then 
\[
\int_{\R^3} (v\cdot \nabla_x Pf)\bmx
\sqrt\mu\\ v_i\sqrt\mu\\\frac{|v|^2-3}{2}\sqrt\mu\\
\emx dv=\bmx 
\nabla_x\cdot b\\
\pt_i(a+c)\\
\nabla_x\cdot b\\
\emx .
\]
\end{lemma}
\begin{proof}
We have 
\[\bega 
&\lw \la v\cdot\nabla_x Pf, \bmx \sqrt\mu\\ v_i\sqrt\mu\\\frac{|v|^2-3}{2}\sqrt\mu\\\emx 
\rw\ra=\lw \la v_k\pt_k(a+b_mv_m+c\frac{|v|^2-3}{2})\sqrt\mu, \bmx \sqrt\mu\\ v_i\sqrt\mu\\\frac{|v|^2-3}{2}\sqrt\mu\\\emx 
\rw\ra\\
&=\bmx \pt_k b_m \la \hat A_{km}+\frac 1 3 \delta_{km}|v|^2\sqrt\mu,\sqrt\mu\ra\\
\pt_k a\delta_{ik}+\pt_k c \la\hat A_{ki}+\frac 1 3 \delta_{ki}|v|^2\sqrt\mu,\frac{|v|^2-3}{2}\ra\\
\pt_k b_m\la \hat A_{km}+\frac 1 3 \delta_{km}|v|^2\sqrt\mu,\frac{|v|^2-3}{2}\sqrt\mu\ra
\emx =\bmx \div b\\\pt_i a+\pt_i c\\\div b
\emx 
\enda 
\]
where we used the fact that $\int_{\R^3}(|v|^2-3)^2\mu dv=6$. The proof is complete.
\end{proof}
\begin{lemma} \label{v-dot-L1} If $f\in \mathcal N^\perp$, then 
there holds 
\[
\int_{\R^3} (v\cdot\nabla_x L^{-1}f)\bmx \sqrt\mu\\ v_i\sqrt\mu\\ \frac{|v|^2-3}{2}\sqrt\mu\\
\emx dv =\bmx 
0\\
\sum_j \la A_{ij},\pt_{x_j}f\ra\\
\sum_j \la B_j,\pt_{x_j}f\ra
\emx .
\]
Here we recall $A_{ij}$ and $B_j$ in \eqref{Aij} and \eqref{Bj}.
\end{lemma} 
\begin{proof}
We have 
\[\bega 
&\lw \la v\cdot\nabla_x L^{-1}f, \bmx \sqrt\mu\\ v_i\sqrt\mu\\\frac{|v|^2-3}{2}\sqrt\mu\\\emx 
\rw\ra=\bmx 
\la L^{-1} \pt_j f,v_j\sqrt\mu\ra\\
\la L^{-1}\pt_k f,v_j v_i\sqrt\mu\ra\\
\la L^{-1}\pt_j f, v_j\frac{|v|^2-3}{2}\sqrt\mu\ra\\
\emx .
\enda 
\]
The results follow, using the fact that $L^{-1}$ is self-adjoint and $L^{-1}:\mathcal N^\perp\to \mathcal N^\perp$.
\end{proof}
\begin{lemma}\label{Gammafg} For $f,g\in \mathcal N$, there holds
\[
L^{-1}\lw(\Gamma(f,g)\rw)=(I-P)\lw\{\frac{fg}{2\sqrt\mu}\rw\}.
\]
\end{lemma}
\begin{proof}
It suffices to show that 
$
L^{-1}\Gamma(f,f)=(I-P)\frac{f^2}{2\sqrt\mu}
$.
Expanding in $s$ the identity 
\[
0=Q\lw(\mu e^{\frac {sf}{\sqrt\mu}},\mu e^{\frac{sf}{\sqrt\mu}}\rw)
\]
and comparing $s^2$ term on both sides, we get the result.
\end{proof}
\begin{lemma} \label{convect} Let $f\in \mathcal N$ and $f=(a_f+b_f\cdot v+c_f\frac{|v|^2-3}{2})\sqrt\mu$, for fix $i$, there holds 
\[
\la v\cdot\nabla_x L^{-1}\Gamma(f,f),v_i\sqrt\mu\ra=b_f\cdot\nabla_x b_f^i -\frac 1 3 \pt_i(|b_f|^2).
\]
\end{lemma}
\begin{proof} Since  $f\in \mathcal N$, by Lemma \ref{fg-lem},  we have 
\[\bega
L^{-1}(\Gamma (f,f))&=\frac 1 2(I-P)\lw(\frac{f^2}{\sqrt\mu}
\rw)=\frac 1 2  (I-P)\lw\{\lw(b_f\cdot v+c_f\frac{|v|^2-3}{2}\rw)^2\sqrt\mu
\rw\}.
\enda 
\]
This implies 
\[\bega 
\la v\cdot\nabla_x L^{-1}\Gamma(f,f),v_i\sqrt\mu\ra &=\frac{1}{2}\sum_j \lw\la \pt_j \lw\{(I-P)\lw[b_f\cdot v+c_f\frac{|v|^2-3}{2}\rw]^2\sqrt\mu\rw\}, v_j v_i\sqrt\mu
\rw\ra \\
&=\frac 1 2 \sum_j \lw\la \pt_j \lw\{(I-P)\lw[b_f\cdot v+c_f\frac{|v|^2-3}{2}\rw]^2\sqrt\mu\rw\}, \hat A_{ij}\rw\ra\\
&=\frac 1 2 \sum_j \lw\la \pt_j \lw\{\lw[b_f\cdot v+c_f\frac{|v|^2-3}{2}\rw]^2\sqrt\mu\rw\}, \hat A_{ij}
\rw\ra\\
&=\frac 1 2 \sum_j \lw\la \pt_j\lw((b_f\cdot v)^2\sqrt\mu
\rw),\hat A_{ij}
\rw\ra=b_f\cdot\nabla_x b_f^i -\frac 1 3 \pt_i (|b_f|^2).
\enda 
\]
The proof is complete.
\end{proof}
\begin{lemma} \label{fg-lem}For $f,g\in \mathcal N$, let
\[
f=\lw(a_f+b_f\cdot v+c_f\frac{|v|^2-3}{2}\rw)\sqrt\mu,\qquad g=\lw(a_g+b_g\cdot v+c_g\frac{|v|^2-3}{2}\rw)\sqrt\mu.\]
There holds \[
\begin{cases}
&\sum_j \la A_{ij},\pt_j \Gamma(f,g)\ra=\frac 1 2 \eta_0 \lw(\div(b_f^i b_g )+\div( b_g^ib_f)-\frac 2 3 \pt_i (b_f\cdot b_g)\rw),\\
&\sum_j \la B_j,\pt_j\Gamma (f,g)\ra=\frac 1 2 \eta_c \div(c_g b_f+c_f b_g),\\
&\sum_j \la A_{ij},\kappa\pt_j (v\cdot\nabla_x g)\ra=\kappa\eta_0\triangle b_g^i+\frac 1 3 \kappa\eta_0\pt_i\div(b_g),\\
&\sum_j \la B_{j},\kappa\pt_j (v\cdot\nabla_x g)\ra=\kappa \eta_c \triangle c_g.
\end{cases}\]
Here $\eta_0,\eta_c>0$ are universal constants defined in \eqref{BjBk} and \eqref{AijAkl}.
\end{lemma}
\begin{proof}
Since $f,g\in \mathcal  N$, by Lemma \ref{Gammafg}, we have \[\bega 
\Gamma(f,g)&=L(I-P)\frac{fg}{2\sqrt\mu}=\frac 1 2 L(I-P)\lw\{
\lw(a_f+b_f\cdot v+c_f\frac{|v|^2-3}{2}\rw)\lw(a_g+b_g\cdot v+c_g\frac{|v|^2-3}{2}\rw)\sqrt\mu\rw\}\\
&=\frac 1 2 L(I-P)\lw\{\lw(b_f\cdot v+c_f\frac{|v|^2-3}{2}\rw)\lw(b_g\cdot v+c_g\frac{|v|^2-3}{2}\rw)\sqrt\mu\rw\}.
\enda \]
Hence 
\[\bega 
&\sum_j \la A_{ij},\pt_j \Gamma(f,g)\ra\\
&=\frac 12\sum_j \lw\la LA_{ij},\pt_j\lw\{\lw(b_f\cdot v+c_f\frac{|v|^2-3}{2}\rw)\lw(b_g\cdot v+c_g\frac{|v|^2-3}{2}\rw)\sqrt\mu\rw\}\rw\ra\\
&=\frac 1 2\la \hat A_{ij},\pt_j(b_f^k b_g^m) v_kv_m\sqrt\mu\ra+\frac 1 2 \pt_j(c_gb_f^k) \la \hat A_{ij},\hat B_k\ra)\\
&\quad+\frac 1 2 \pt_j(c_fb_g^k)\la \hat A_{ij},\hat B_k\ra+\frac 1 8 \pt_j(c_fc_g)\la \hat A_{ij},(|v|^2-3)^2\sqrt\mu\ra\\
&=\frac 1 2 \pt_j(b_f^kb_g^m) \sum_j \la \hat A_{ij},A_{km}\ra=\frac 1 2 \eta_0\pt_j(b_f^kb_g^m)\lw(\delta_{ik}\delta_{jm}+\delta_{im}\delta_{jk}-\frac 2 3\delta_{ij}\delta_{km}
\rw)\\
&=\frac 1 2 \eta_0 \lw(\div(b_f^i b_g )+\div( b_g^ib_f)-\frac 2 3 \pt_i (b_f\cdot b_g)\rw).
\enda 
\]
Now we show the second identity. Similarly, we have 
\[\bega 
&\sum_j \la B_j,\pt_j \Gamma(f,g)\ra\\
&=\frac 1 2\la \hat B_j,\pt_j(b_f^k b_g^m) v_kv_m\sqrt\mu\ra+\frac 1 2 \pt_j(c_gb_f^k) \la \hat B_j ,\hat B_k\ra)\\
&\quad+\frac 1 2 \pt_j(c_fb_g^k)\la \hat B_j,\hat B_k\ra+\frac 1 8 \pt_j(c_fc_g)\la \hat B_j ,(|v|^2-3)^2\sqrt\mu\ra\\
&=\frac 1 2 \pt_j(c_g b_f^k) \eta_c \delta_{jk}+\frac 12 \eta_c\pt_j(c_f b_g^k)\delta_{jk}\\
&=\frac 1 2 \eta_c \div(c_g b_f+c_f b_g).
\enda \]
Now we have 
\[\bega 
&\sum_j \la A_{ij},\kappa\pt_j (\pt_t f+v\cdot\nabla_x g)\ra=\kappa \sum_j\lw \la A_{ij},v_k\pt_{jk} \lw(a_g+b_g^m v_m+c_g\frac{|v|^2-3}{2}\rw)\sqrt\mu\rw\ra\\
&=\kappa\sum_j \pt_{jk}b_g^m \la A_{ij},\hat A_{km}\ra=\kappa\eta_0 \pt_{jk}b_g^m \lw(\delta_{ik}\delta_{km}+\delta_{im}\delta_{jk}-\frac 2 3\delta_{ij}\delta_{km}
\rw)=\kappa\eta_0(\triangle b_g^i+\frac 1 3\pt_i \div(b_g)),
\enda 
\]
and 
\[\bega 
&\sum_j \la B_j,\kappa\pt_j (\pt_t f+v\cdot\nabla_x g)\ra=\kappa \sum_j\lw \la B_j,v_k\pt_{jk} \lw(a_g+b_g^m v_m+c_g\frac{|v|^2-3}{2}\rw)\sqrt\mu\rw\ra\\
&=\kappa\sum_j \pt_{jk}c_g \la B_j,\hat B_{k}\ra=\kappa \eta_c \triangle c_g.
\enda 
\]
where we used \eqref{BjBk}.

\end{proof}

\begin{lemma}\label{gap-lem} Let $d_T$ be the constant defined in \eqref{gap1}. 
There exists $\tau_\star>0$ such that for $\kappa$ small, we have 
\[\bega 
|y_i^\kappa(\tau)-y_i(\tau)|\lesssim e^{-\frac{d_T^2}{4\kappa}},\qquad \text{and}\quad 
\min_{i\neq j}|y_i^\kappa(\tau)-y_j^\kappa(\tau)| \ge d_T.
\enda 
\]
for all $i\in [1,N]$ and $\tau\in [0,\tau_\star]$. Here we lower the value of $d_T$ so that both \eqref{gap1}  and the above inequalities hold.
\end{lemma}
\begin{proof}
Assume that 
\beq \label{gap}
\min_{i\neq j}|y_i^\kappa(\tau)-y^\kappa_j(\tau)|\ge d_T,
\eeq we have 
\[\bega 
&\max_{1\le i\le N}|y_i^\kappa(\tau)-y_i(\tau)|\\
&\lesssim \max_{1\le i\le N}\int_0^\tau |\pt_{\wtd \tau} y_i^\kappa-\pt_{\wtd \tau} y_i|d\wtd \tau\\
&\lesssim \max_{i\neq j}\int_0^\tau \lw|
K_{B}(y_{ij}^\kappa(\wtd\tau))-K_{B}(y_{ij}(\wtd\tau))
\rw| d\wtd \tau+\max_{i\neq j}\int_0^\tau |K_{B}(y_{ij}^\kappa(\wtd\tau))|e^{-\frac{|y_{ij}^\kappa(\wtd \tau)|^2}{4\kappa e^\tau}}d\wtd \tau\\
&\lesssim \max_{i\neq j} \int_0^\tau |y_{ij}^\kappa(\wtd \tau)-y_{ij}(\wtd\tau)|d\wtd \tau\cdot \lw(d_T^{-2}+d_T^{-3}\rw)+d_T^{-2}e^{-\frac{d_T^2}{4\kappa}}\\
&\lesssim \int_0^\tau \max_{1\le i\le N}|y_i^\kappa(\wtd \tau)-y_i(\wtd \tau)|d\wtd \tau+e^{-\frac{d_T^2}{4\kappa}}.
\enda 
\]
By the Gronwall inequality, we have 
\[
\max_{1\le i\le N}|y_i^\kappa(\tau)-y_i(\tau)|\lesssim e^{-\frac{d_T^2}{4\kappa}}.
\]
The condition \eqref{gap} in ensured, using the fact that $\min_{i\neq j}|y_i(\tau)-y_j(\tau)|\ge d_T$.
The proof is complete, by a Gronwall inequality. 
\end{proof}
\begin{lemma} \label{elliptic1} Let $u=K\star \w$ be the velocity vector field obtained from the vorticity $\w$ on $\R^2$. Define the norm $\|\cdot\|_{\lbb}=\|\cdot\|_{L^4}+\|\cdot\|_{L^{4/3}}$. There hold the following inequalities
\[
\bega
\|u\|_{L^\infty}&\lesssim \|\w\|_{\lbb},\qquad \|u\|_{L^\infty}\lesssim \|\w\|_{L^1\cap L^\infty} .
\enda
\]
Moreover, if $\int_{\R^2}\w(x)dx=0$, then 
\[
\| (1+|x|^2)v\|_{L^\infty} \lesssim \|(1+|x|^2)\w\|_{\lbb}.
\]
\end{lemma}
\begin{proof}
From the Biot-Savart law, we get
\beq \label{repeat}
\bega
|u(x)|&\lesssim \int_{\R^2}\frac{|\w(x')|}{|x-x'|}dy=\lw(\int_{|x-x'|\le \bar M}+\int_{|x-x'|\ge \bar M}\rw)\frac{|\w(x')|}{|x-x'|}dx',\\
&\lesssim \lw(\int_{|x-x'|\le \bar M}|x-x'|^{-4/3}dy\rw)^{3/4}\|\w\|_{L^4}+\lw(\int_{|x-x'|\ge \bar M}|x-x'|^{-4}dx'\rw)^{1/4}\|\w\|_{L^{4/3}},\\
&\lesssim \bar M^{1/2}\|\w\|_{L^4}+\bar M^{-1/2}\|\w\|_{L^{4/3}}.\\
\enda
\eeq
Thus choosing $\bar M=\frac{\|\w\|_{L^{4/3}}}{\|\w\|_{L^4}}$, we have $\|u\|_{L^\infty}\lesssim \|\w\|_{L^{4/3}}^{1/2}\|\w\|_{L^4}^{1/2} $, which gives the first inequality. For the second inequality, we use $\| \omega \|_{L^{\bar p}} \le \| \omega \|_{L^1}^{\frac 1 {\bar p}}\|\omega\|^{1-\frac 1 {\bar p}}_{L^\infty}$. 
It remains to check the last inequality. We shall check it only for $u_2$, the second component of $u$. First, we check \beq \label{est20}
|x||u_2(x)|\lesssim \int_{\R^2}\frac{1}{|x-x'|}|x'||\w(x')|dx'.
\eeq
By Biot-Savart law and $\int_{\R^2}\w(x')dx'=0$, we have 
\[
|u_2(x)|=\frac{1}{2\pi}\lw|\int_{\R^2}\frac{x_1-x'_1}{|x-x'|^2}\w(x')dx'\rw|\lesssim \int_{\R^2}\lw|\frac{x_1-x'_1}{|x-x'|^2}-\frac{x_1}{|x|^2}\rw||\w(x')|dx'.\\
\]
Now we have 
\[
\frac{x_1-x'_1}{|x-x'|^2}-\frac{x_1}{|x|^2}=\frac{1}{|x|^2|x-x'|^2}\lw(|x|^2(x_1-x'_1)-x_1|x-x'|^2\rw).
\]
It follows that $|x|^2(x_1-x'_1)-x_1|x-x'|^2\le 4|x||x'||x-x'|$. Hence, 
\[
|x|\Big[ \frac{x_1-x'_1}{|x-x'|^2}-\frac{x_1}{|x|^2} \Big]\le  \frac{4|x'|}{|x-x'|},
\]
which gives \eqref{est20}. Now multiplying both sides of \eqref{est20} by $|x|$, we have 
\[\bega
|x|^2|u_2(x)|&\lesssim \int_{\R^2}\frac{|x||x'|}{|x-x'|}|\w(x')|dx'\le \int_{\R^2}\frac{|x'|+|x-x'|}{|x-x'|}|x'||\w(x')|dx'.\\
&=\int_{\R^2}\frac{1}{|x-x'|}|x'|^2|\w(x')|dx'+\int_{\R^2}|x'||\w(x')|dx'.\\
\enda
\]
Let us first treat the first term in the above. Repeating the argument of \eqref{repeat} for $\w=|x'|^2|\w(x')|$, we have
\[
\int_{\R^2}\frac{1}{|x-x'|}|x'|^2|\w(x')|dx'\lesssim \|(1+|x'|^2)\w(x')\|_{\lbb}.
\]
For the second term, using H\"older inequality, we obtain
\[
\int_{\R^2}|x'||\w(x')|dx'=\int_{\R^2}\frac{|x'|}{1+|x'|^2}(1+|x'|^2)||\w(x')|dx'\lesssim\|(1+|x'|^2)|\w(x')\|_{L^{4/3}}.
\]
Thus 
\[
|x|^2|u_2(x)|\lesssim \|(1+|x|^2)\w\|_{\lbb}.
\]
The proof is complete.
\end{proof}
\begin{lemma} Let $\mathfrak m(v)=e^{\rho_0|v|^2}$ where $\rho_0>0$.
There holds 
\beq\label{K-ineq1}
\int_{\R^3}|K(v,v_\star)|\frac{\mathfrak m(v)}
{\mathfrak m(v_\star)}(1+|v-v_\star|^{N_0})dv_\star\lesssim \frac{1}{1+|v|}
\eeq
for all $\rho_0>0$ sufficiently small and for any $N_0>0$. As a consequence,
\[
\int_{\R^3}|K(v,v_\star)| \frac{\mathfrak m'(v)}{\mathfrak m'(v_\star)}dv_\star\lesssim 1
\]
where $\mathfrak m'(v)=\frac{\mathfrak m(v)}{\nu(v)}$.
\end{lemma}
\begin{proof} From the Grad estimate \eqref{grad}, it suffices to show the inequalities \eqref{K-ineq1} for $K=K_\vartheta$ defined in \eqref{grad}. Let $q=v-v_\star$. 
We have 
\beq\label{vq1}
\bega 
-|v-v_\star|^2-\frac{(|v|^2-|v_\star|^2)^2}{|v-v_\star|^2}&=-2|q|^2+4v\cdot q-4\frac{(v\cdot q)^2}{|q|^2}\\
&\le -2|q|^2+4\lw(\eps_1 |q|^2+\frac 1 {4\eps_1}\frac{|v\cdot q|^2}{|q|^2}
\rw)-4\frac{|v\cdot q|^2}{|q|^2}\\
&=-2(1-2\eps_1)|q|^2-4(1-(4\eps_1)^{-1})\frac{|v\cdot q|^2}{|q|^2}
\enda 
\eeq
for any $\eps_1>0$. On the other hand, we have 
\beq\label{vq2}
\bega 
\frac{\mathfrak m(v)}{\mathfrak m(v_\star)}&=e^{\rho_0(|v|^2-|v_\star|^2)}= e^{\rho_0(|v|^2-|v-q|^2)}=e^{-\rho_0|q|^2+2\rho_0 (v\cdot q)},\\
&\le e^{-\rho_0|q|^2+2\rho_0\eps_2 |q|^2+2\rho_0\frac{1}{4\eps_2}\frac{|v\cdot q|^2}{|q|^2}}=e^{-\rho_0(1-2\eps_2)|q|^2+\frac{\rho_0}{2\eps_2}\frac{|v\cdot q|^2}{|q|^2}}
\enda 
\eeq
for any $\eps_2>0$. 
Combining \eqref{vq1} and \eqref{vq2}, we have 
\[\bega
K_{\vartheta}(v,v_\star)\frac{\mathfrak m(v)}{\mathfrak m(v_\star)}(1+|v-v_\star|^{N_0})&\le \frac{1}{|q|}e^{-(2\vartheta(1-2\eps_1)+\rho_0(1-2\eps_2))|q|^2}e^{-\lw(4(\vartheta(1-(4\eps_1)^{-1})-\frac{\rho_0}{4\eps_2}\rw)\frac{|v\cdot q|^2}{|q|^2}}(1+|q|^{N_0})\\
&\lesssim \frac 1 {|q|}e^{-c_0 |q|^2-c_0 |v\cdot q|^2}
\enda 
\]
for some $c_\theta>0$.
Hence 
\[\bega 
&\int_{\R^3}K_{\vartheta}(v,v_\star)\frac{\mathfrak m(v)}{\mathfrak m(v_\star)}(1+|v-v_\star|^{N_0})dv_\star\\
&\lesssim \int_{\R^3}\frac 1 {|q|}e^{-c_\theta |q|^2-c_\theta|v\cdot q|^2}dq\\
&\lesssim \int_0^\infty re^{-c_0 r^2}\int_0^\pi e^{-c_0 |v|r\cos \theta}\sin\ta d\ta dr\lesssim \frac 1 {1+|v|}.
\enda 
\]
The proof is complete.
\end{proof}
\begin{lemma}
\label{BiGamma} For any $f,g,h\in L^2_v$, there hold 
\beq\label{Bi1}
\la\Gamma(f,g), h\ra_{L^2_{v}}\lesssim\lw\{ \|\mathfrak mf\|_{L^\infty_{v}}\|\sqrt\nu  g\|_{L^2_{v}}+\|\mathfrak mg\|_{L^\infty_{v}}\|\sqrt\nu f\|_{L^2_{x,v}}\rw\}\|\sqrt\nu (I-P) h\|_{L^2_{v}}.
\eeq
\beq\label{Bi2}
\la \DF \Gamma(f,g), h\ra_{L^2_{x,v}}\lesssim\lw\{ \|\mathfrak mf\|_{L^\infty_{x,v}}\|\sqrt\nu \DF g\|_{L^2_{x,v}}+\|\mathfrak mg\|_{L^\infty_{x,v}}\|\sqrt\nu \DF f\|_{L^2_{x,v}}\rw\}\|\sqrt\nu (I-P) h\|_{L^2_{x,v}}.
\eeq
and 
\beq\label{Bi3}
\bega
\lw\| \Gamma(f,g)\sqrt\varphi 
\rw\|_{L^2_{v}}&\lesssim \|f\|_{L^\infty_{v}}\| g\|_{L^2_{v}}+\|g\|_{L^\infty_{v}}\| f\|_{L^2_{v}},\\
\lw\|\DF \Gamma(f,g)\sqrt\varphi 
\rw\|_{L^2_{x,v}}&\lesssim  \|f\|_{L^\infty_{x,v}}\|\DF g\|_{L^2_{x,v}}+\|g\|_{L^\infty_{x,v}}\|\DF f\|_{L^2_{x,v}}.\\
\enda 
\eeq
where $\varphi\in \mathfrak B=\{\sqrt\mu, v\sqrt\mu, \frac{|v|^2-3}{2}\sqrt\mu\}$.
\end{lemma}
\begin{proof}
We give the proof for \eqref{Bi2} only, since \eqref{Bi1} is similar. First we recall that $\Gamma(f,g)=\Gamma_{\text{gain}}(f,g)-\Gamma_{\text{loss}}(f,g)$ where 
\[\bega
\Gamma_{\text{gain}}(f,g)(v)&=\Gamma_{\text{gain,1}}(f,g)+\Gamma_{\text{gain,2}}(f,g),\\
\Gamma_{\text{loss}}(f,g)(v)&=\Gamma_{\text{loss,1}}(f,g)+\Gamma_{\text{loss,2}}(f,g),
\enda 
\]
and 
\beq\label{gainloss}
\begin{cases}
\Gamma_{\text{gain,1}}(f,g)&=\int_{\R^3\times\mathbb S^2}|(v-v_\star)\cdot \wtd s|\sqrt{\mu(v_\star)}f(v')g(v_\star')d\wtd s dv_\star,\\
\Gamma_{\text{gain,2}}(f,g)&=\int_{\R^3\times\mathbb S^2}|(v-v_\star)\cdot \wtd s|\sqrt{\mu(v_\star)}g(v')f(v_\star')d\wtd s dv_\star,\\
\Gamma_{\text{loss,1}}(f,g)&=\int_{\mathbb R^3\times \mathbb S^2}|(v-v_\star)\cdot \wtd s|\sqrt{\mu(v_\star)}f(v_\star)g(v)d\wtd s dv_\star,\\
\Gamma_{\text{loss,2}}(f,g)&=\int_{\mathbb R^3\times \mathbb S^2}|(v-v_\star)\cdot \wtd s|\sqrt{\mu(v_\star)}g(v_\star)f(v)d\wtd s dv_\star.\\
\end{cases}
\eeq
We shall show the bound for $\Gamma_{\text{gain,1}}(f,g)$ only, since $\Gamma_{\text{gain,2}}$ is done similarly. We have 
\[\bega 
&\la \DF \Gamma_{\text{gain,1}}(f,g),h\ra_{L^2_v}\\
&\lesssim \int_{\R^3}\lw\{\int_{\R^3\times \mathbb S^2}|(v-v_\star)\cdot \wtd s| \sqrt{\mu(v_\star)}\DF \lw(f(v')g(v_\star')\rw)d\wtd s dv_\star\rw\} (I-P)h dv\\
&\lesssim  \int_{\R^3}\lw\{\nu(v) \lw(\int_{\R^3}|\DF (f(v')g(v_\star'))|^2dv_\star\rw)^{1/2}\rw\}|(I-P)h |dv\\
&\lesssim \lw(\int_{\R^3\times\R^3 }\nu(v)\lw|\DF (f(v')g(v_\star'))\rw|^2dv_\star dv\rw)^{1/2} \|\sqrt\nu (I-P)h\|_{L^2_v}.
\enda 
\]
For the first factor above, we have 
\beq
\bega 
&\int_{\R^2}\int_{\R^3}\int_{\R^3}\nu(v)\lw|\DF (f(v')g(v_\star'))\rw|^2dv_\star dvdx\\
&\lesssim\int_{\R^3\times\R^3}\nu(v)\lw\{ \|f(v',\cdot)\|^2_{L^\infty_x}\|\DF g(v_\star',\cdot)\|^2_{L^2_x}+ \|g(v'_\star,\cdot)\|^2_{L^\infty_x}\|\DF f(v',\cdot)\|^2_{L^2_x}\rw\}dv_\star dv\\
&\lesssim \int_{\R^3\times\R^3}\lw\{\nu(v')+\nu(v_\star')
\rw\}
\lw\{ \|f(v',\cdot)\|^2_{L^\infty_x}\|\DF g(v_\star',\cdot)\|^2_{L^2_x}+ \|g(v'_\star,\cdot)\|^2_{L^\infty_x}\|\DF f(v',\cdot)\|^2_{L^2_x}\rw\}dv'_\star dv'\\
&\lesssim \|\mathfrak m(v) f\|_{L^\infty_{x,v}}^2 \lw\|\sqrt\nu \DF g\rw\|_{L^2_{x,v}}^2+\|\mathfrak m(v) g\|_{L^\infty_{x,v}}^2\|\sqrt\nu \DF f\|_{L^2_{x,v}}^2.
\enda 
\eeq
Hence we get 
\[
\la \DF \Gamma_{\text{gain,1}}(f,g),h\ra_{L^2_{x,v}}\lesssim \lw\{\|\mathfrak m f\|_{L^\infty_{x,v}}\|\sqrt\nu g\|_{L^2_{x,v}}+\|\mathfrak m g\|_{L^\infty_{x,v}}\|\sqrt\nu \DF f\|_{L^2_{x,v}}
\rw\} \|(I-P)h\|_{L^2_{x,v}}.
\]
Next we get 
\beq\label{Gamma-ineq3}
\bega 
&\la \DF \Gamma_{\text{loss,1}}(f,g),h\ra_{L^2_v}\\
&\lesssim \int_{\R^3} \nu(v)\lw\{\int_{\R^3}|\DF(f(v_\star)g(v))|^2dv_\star \rw\}^{1/2}|(I-P)h(v)|dv\\
&\lesssim \lw\{\int_{\R^3}\nu(v)|\DF f(v_\star)g(v)|^2dv_\star dv\rw\}^{1/2}\|\sqrt\nu (I-P)h\|_{L^2_v}.
\enda
\eeq
Integrating in $x$, we get
\[\bega 
&\la \DF \Gamma_{\text{loss,1}}(f,g),h\ra_{L^2_{x,v}}\\
&\lesssim \lw\{\int_{\R^2}\int_{\R^3}\nu(v)|\DF(f(v_\star)g(v))|^2dv_\star dv dx\rw\}^{1/2}\|\sqrt\nu (I-P)h\|_{L^2_{v}}.
\enda
\]
For the first factor, we obtain 
\beq\label{Gamma-ineq4}
\bega 
&\int_{\R^2}\int_{\R^3}\int_{\R^3}\nu(v)\lw|\DF (f(v_\star)g(v))\rw|^2dv_\star dvdx\\
&\lesssim\int_{\R^3\times\R^3}\nu(v)\lw\{ \|g(v,\cdot)\|^2_{L^\infty_x}\|\DF f(v_\star,\cdot)\|^2_{L^2_x}+ \|f(v_\star,\cdot)\|^2_{L^\infty_x}\|\DF g(v,\cdot)\|^2_{L^2_x}\rw\}dv_\star dv\\
&\lesssim \|\mathfrak m g\|_{L^\infty_{x,v}}^2 \lw\|\sqrt\nu \DF f\rw\|_{L^2_{x,v}}^2+\|\mathfrak m f\|_{L^\infty_{x,v}}^2\|\sqrt\nu \DF g\|_{L^2_{x,v}}^2.
\enda 
\eeq
Hence we obtain 
\[
\la \DF \Gamma_{\text{loss,1}}(f,g),h\ra_{L^2_{x,v}}\lesssim \lw\{\|\mathfrak m g\|_{L^\infty_{x,v}}\|\sqrt\nu \DF f\|_{L^2_{x,v}}+\|\mathfrak m f\|_{L^\infty_{x,v}}\|\sqrt\nu \DF g\|_{L^2_{x,v}}
\rw\}\|(I-P)h\|_{L^2_{x,v}}.
\]
This finishes the proof for the inequality \eqref{Bi2}.
Next we show \eqref{Bi3}. Take $h\in L^2_{x,v}$ so that $\|h\|_{L^2_{x,v}}\le 1$. We need to show that 
\[
\lw\la \DF \Gamma(f,g)\sqrt\varphi , h \rw\ra_{L^2_{x,v}}\lesssim \|f\|_{L^\infty_{x,v}}\|\DF g\|_{L^2_{x,v}}+\|g\|_{L^\infty_{x,v}}\|\DF f\|_{L^2_{x,v}}.
\]
This is obvious from the proof of \eqref{Bi2} presented above, along with the fact that $\nu(v)\varphi(v)$ is decaying fast when $|v|\to \infty$. We skip the details. The proof is complete.
\end{proof}
\begin{lemma}\label{Gamma-inf}
Recall that $\mathfrak m'(v)=\frac{\mathfrak m(v)}{\nu(v)}$. 
There holds 
\[
\|\mathfrak m' \Gamma(f,g)\|_{L^\infty_v}\lesssim \|\mathfrak m f\|_{L^\infty_v}\|\mathfrak m g\|_{L^\infty_v}.
\]
\end{lemma}
\begin{proof}
It is straightforward that 
\[
\mathfrak m(v) \Gamma_{\text{loss}}(f,g)(v)\lesssim \nu(v) \mathfrak m(v) f(v)\|\mathfrak mg\|_{L^\infty_v}+\|\mathfrak m f\|_{L^\infty_v}\|\mathfrak mg\|_{L^\infty_v}.
\]
As for $\Gamma_{\text{gain,1}}$ defined in \eqref{gainloss}, we have $\mathfrak m(v)\lesssim \mathfrak m(v')\mathfrak m(v_\star')$. This implies 
\[\bega 
\mathfrak m(v)\Gamma_{\text{gain},1}(f,g)&\lesssim \int_{\R^3}|v-v_\star|\sqrt{\mu(v_\star)}\lw|\mathfrak m f(v') \rw|\cdot |\mathfrak m g(v_\star')|dv_\star\\
&\lesssim \|\mathfrak m f\|_{L^\infty_v}\|\mathfrak m g\|_{L^\infty} \nu(v).
\enda 
\]
Similarly, we get $\mathfrak m(v)\Gamma_{\text{gain,2}}(f,g)(v)\lesssim \nu(v)\|\mathfrak m f\|_{L^\infty_v}\|\mathfrak m g\|_{L^\infty} $. The proof is complete.
\end{proof}

\bibliographystyle{abbrv}
\def\cprime{$'$} \def\cprime{$'$}

\end{document}